\documentclass[12pt]{amsart}

\usepackage{amsfonts}
\usepackage{amsmath}
\usepackage{amssymb}
\usepackage[margin=1in]{geometry}
\usepackage{amsthm}    
\usepackage{xcolor}

\usepackage{mathtools}							
\usepackage{commath}							
							 
\usepackage{enumitem}  							
											
\usepackage[overload]{textcase}					
\usepackage{array}								
\usepackage{booktabs}							

\usepackage[english]{babel}						

\usepackage[hidelinks]{hyperref}

\usepackage{bookmark}
\usepackage{color,soul}
\soulregister\cite7

\usepackage{bigints}

\usepackage{graphicx}

\usepackage{relsize}

\makeatletter
\@namedef{subjclassname@2020}{\textup{2020} Mathematics Subject Classification}
\makeatother

\newtheorem{theorem}{Theorem}[section]

\newtheorem{lemma}[theorem]{Lemma}
\newtheorem{proposition}[theorem]{Proposition}

\theoremstyle{definition}
\newtheorem{definition}[theorem]{Definition}
\newtheorem{remark}[theorem]{Remark}

\newtheorem{open problem}[theorem]{Open Problem}

\numberwithin{equation}{section} 

\newcommand{\N}{\mathbb{N}}
\newcommand{\Z}{\mathbb{Z}}

\newcommand{\R}{\mathbb{R}}

\renewcommand{\Pr}{\operatorname{P}}

\newcommand{\defeq}{\coloneqq}
\newcommand{\eqdef}{\eqqcolon}

\begin{document}

\title[J\MakeLowercase{oint distribution of the sum and maximum of independent random variables}]{R\MakeLowercase{ecurrence relations for the joint distribution of the sum and maximum of independent random variables}}

\author[C\MakeLowercase{hristos}~N.~E\MakeLowercase{frem}]{C\MakeLowercase{hristos}~N.~E\MakeLowercase{frem}}

\address{Department of Electrical and Computer Engineering, \newline \indent University of Cyprus, Nicosia, Cyprus}
\email{efrem\,\{dot\}\,christos\,\{at\}\,ucy\,\{dot\}\,ac\,\{dot\}\,cy}

\begin{abstract}
In this paper, the joint distribution of the sum and maximum of independent, not necessarily identically distributed, nonnegative random variables is studied for two cases: i) continuous and ii) discrete random variables. First, a recursive formula of the joint cumulative distribution function (CDF) is derived in both cases. Then, recurrence relations of the joint probability density function (PDF) and the joint probability mass function (PMF) are given in the former and the latter case, respectively. Interestingly, there is a fundamental difference between the joint PDF and PMF. The proofs are simple and mainly based on the following tools from calculus and discrete mathematics: differentiation under the integral sign (also known as Leibniz's integral rule), the law of total probability, and mathematical induction. In addition, this work generalizes previous results in the literature, and finally presents several extensions of the methodology.   
\end{abstract}

\subjclass[2020]{Primary 60E05, 60G50, 60G70; Secondary 26B05, 65Q30}
\keywords{Probability distributions, sum and maximum of independent random variables, recurrence relations, differentiation under the integral sign, Dirac delta function, peak-to-average ratio}

\maketitle

\section{Introduction}

In the past decades there has been a growing interest in the joint distribution of the sum and maximum of independent random variables, mainly due to its practical importance in several fields (e.g., engineering, geosciences and economics) as well as its own theoretical significance. Nevertheless, general mathematical approaches and algorithmic methods to deal with arbitrary distributions of random variables are still missing from the literature.

\subsection{Related Work}

Chow and Teugels \cite{Chow-Teugels},  Anderson and Turkman \cite{Anderson-Turkman1991,Anderson-Turkman1993}, Hsing \cite{Hsing_a,Hsing_b}, Ho and Hsing \cite{Ho-Hsing}, McCormick and Qi \cite{McCormick-Qi}, and Peng and Nadarajah \cite{Peng-Nadarajah} studied the limiting joint distribution and asymptotic independence of the sum and maximum of identically distributed random variables. In addition, the exact (non-asymptotic) joint distribution of the sum and maximum of independent random variables was investigated by Qeadan, Kozubowski and Panorska \cite{Qeadan-Kozubowski-Panorska}, and by Arendarczyk, Kozubowski and Panorska \cite{Arendarczyk-Kozubowski-Panorska_a} for identically and non-identically distributed exponential variables, respectively. A key recursive formula for independent and identically distributed (i.i.d.)~nonnegative continuous random variables (with an arbitrary PDF) was derived by Qeadan, Kozubowski and Panorska \cite[Eq.~(2.18)]{Qeadan-Kozubowski-Panorska}. Another recursive formula was given in \cite[Eq.~(23)]{Arendarczyk-Kozubowski-Panorska_a} by appropriately partitioning the domain of the joint CDF. Arendarczyk, Kozubowski and Panorska also investigated the case of Pareto-dependent and identically distributed random variables \cite{Arendarczyk-Kozubowski-Panorska_b}. Moreover, recurrence relations were used for the distribution of the sum of independent continuous/discrete random variables \cite{Moschopoulos, Woodward-Palmer, Butler-Stephens} as well as in order statistics \cite{Balakrishnan, Balakrishnan-Bendre-Malik, David-Nagaraja}. It is worth mentioning that the distribution of the sum of independent (non-identical) uniform random variables was originally proved by Olds \cite{Olds} using induction, and later by Lazar and Almutairi~\cite{Lazar-Almutairi} using the properties of Dirac's delta function. Finally, the sum of gamma and exponential random variables  were studied respectively in \cite{Alouini-Abdi-Kaveh} and \cite{Khuong-Kong}, with applications in wireless communication systems.

\subsection{Main Contributions \& Outline}

Despite extensive work in the past, the joint distribution of the sum and maximum of independent, not necessarily identically distributed, random variables remains almost unexplored (except for the special case of exponential variables \cite{Arendarczyk-Kozubowski-Panorska_a}). The purpose of this work is to provide a new approach to this problem in its full generality, not only for continuous but also for discrete random variables with \emph{arbitrary} distributions. In particular, we derive \emph{recursive formulas} for the joint distributions that are of independent theoretical interest and can also be used for numerical computations, especially when it is difficult to obtain explicit (closed-form) expressions. Moreover, such calculations may be used to check whether available data fit these probability distributions well (also known as goodness-of-fit test). In general, recurrence relations are useful for designing efficient methods, such as iterative algorithms, which have great practical value. Finally, note that the main focus of the paper is on the application of mathematical analysis and discrete mathematics to probability theory, rather than on specific real-world applications (which can be found in other works, e.g., \cite{Qeadan-Kozubowski-Panorska, Arendarczyk-Kozubowski-Panorska_a, Arendarczyk-Kozubowski-Panorska_b}, and the references therein).  

The remainder of this paper is organized as follows. Section~\ref{section:Continuous Random Variables} deals with the joint CDF and PDF in the continuous case, and presents a potential application of peak-to-average ratio in communication systems. Subsequently, Section~\ref{section:Discrete Random Variables} studies the joint CDF and PMF in the discrete case, and provides a similar result for the peak-to-average ratio. Section~\ref{section:Extensions and Generalizations} presents various extensions and generalizations of the proposed approach for random variables that are not necessarily nonnegative, have discontinuous PDFs, or are of a mixed (continuous and discrete) nature. Furthermore, Appendices~\ref{appendix:Partial Differentiation Under the Integral Sign}~and~\ref{appendix:Dirac Delta Function} contain information about partial differentiation under the integral sign and the Dirac delta function, respectively. All proofs are simple and elegant using only elementary techniques, such as the law of total probability, Leibniz's integral rule, and mathematical induction. Last but not least, this paper includes existing results as special cases.

\subsection{Notation} 

By $a \defeq b$, or $b \eqdef a$, we mean that ``$a$ is by definition equal to $b$''. $\R$ is the set of real numbers and $\Z$ is the set of integers. $\R_+ \defeq \{x \in \R : x \geq 0\}$ represents the set of nonnegative real numbers. The set of natural numbers excluding/including zero is denoted by $\N \defeq \{1,2,\dots\}$ and $\N_0 \defeq \{0,1,\dots\}$, respectively. $\lfloor x \rfloor \defeq \max \{k \in \Z : k \leq x \}$ denotes the floor of $x$ (i.e., the greatest integer less than or equal to $x$), while $\lceil x \rceil \defeq \min \{k \in \Z : k \geq x \}$ is the ceiling of $x$ (i.e., the least integer greater than or equal to $x$). The symbols $\lor$ and $\land$ stand for logical disjunction (OR) and logical conjunction (AND), respectively. 

In addition, throughout the paper we consider only real-valued functions. $\delta(x)$ stands for the Dirac delta function/distribution (for more details see Appendix~\ref{appendix:Dirac Delta Function}), while $H(x)$ is the Heaviside step function defined by 
\[
H(x) \defeq \left\{  
\begin{array}{ll}
      1, & \text{if}\ x \geq 0 , \\
      0, & \text{if}\ x < 0   . \\
\end{array} 
\right. 
\] 
Note that the derivative of the Heaviside step function is equal to the Dirac delta function, i.e., $\od{}{x} H(x) = \delta(x)$. 
We also adopt the Iverson bracket notation, that is, 
\[
[Q] \defeq \left\{  
\begin{array}{ll}
      1, &  \text{if statement $Q$ is true} , \\
      0, & \text{if statement $Q$ is false} . \\
\end{array} 
\right. 
\] 
Using this notation, the Heaviside step function can be expressed as $H(x) = [x \geq 0]$. 

Let $(\mathcal{X},\Sigma,\mu)$ be a measure space, where $\mathcal{X} \subseteq \R^n$, $\Sigma$ is a $\sigma$-algebra over $\mathcal{X}$, and \linebreak $\mu : \Sigma \to \R \cup \{-\infty,+\infty\}$ is a measure function. Then, $\displaystyle{f(\mathbf{x}) \mathop{=}^{\text{a.e.}} g(\mathbf{x})}$ means that $f(\mathbf{x}) = g(\mathbf{x})$ almost everywhere on $\mathcal{X}$ (or, alternatively, for almost all/every $x \in \mathcal{X}$), i.e., there exists a (measurable) set $\mathcal{S} \in \Sigma$ with $\mu(\mathcal{S}) = 0$ such that $f(\mathbf{x}) = g(\mathbf{x})$ for all $\mathbf{x} \in \mathcal{X} \setminus \mathcal{S}$. For the purposes of this paper, we use the Lebesgue measure (in the geometric sense), that is, length in $\R$, area in $\R^2$, and volume in $\R^3$.  

Furthermore, we make the following conventions: i) $\sum_{k=l}^{u} {a_k} = 0$ and $\prod_{k=l}^{u} {a_k} = 1$, whenever $u<l$, and ii) when a function $f(\mathbf{x})$ is defined on a set/domain $\mathcal{D} \subseteq \R^n$, we implicitly assume that $f(\mathbf{x}) = 0$ elsewhere (i.e., for all $\mathbf{x} \notin \mathcal{D}$); in fact, we can equivalently define an extended function $\widetilde{f}(\mathbf{x})$ on $\R^n$, that is, $\widetilde{f}(\mathbf{x}) \defeq f(\mathbf{x}) [\mathbf{x} \in \mathcal{D}]$.

\section{Continuous Random Variables}
\label{section:Continuous Random Variables}

We consider a sequence of \emph{independent, not necessarily identically distributed,} nonnegative (absolutely) continuous random variables $X_i \geq 0$, $i \in \N$, with PDF $f_i(x)$ and CDF $F_i(x) \defeq \Pr (X_i \leq x) = \int_0^x {f_i(x') \dif x'}$, $\forall x \in \R_+$. For every $i \in \N$, we assume that the PDF $f_i(x)$ is \emph{continuous} on $\R_+$, therefore $f_i(x) = \od{}{x} F_i(x)$. Observe that $f_i(x)$ is also \emph{bounded} on $\R_+$, since $f_i(x) \geq 0$ and $\int_0^{\infty} {f_i(x) \dif x} = 1$. Hence, $f_i(x)$ is \emph{bounded and piecewise continuous on~$\R$}, with possible discontinuity at $x=0$, because $f_i(x) = 0$ for all $x<0$. Now, let us define the following random variables for all $n \in \N$  
\begin{equation} \label{equation:sum and maximum} 
Y_n \defeq \sum_{i=1}^n {X_i} = X_1 + \cdots + X_n , \qquad  Z_n \defeq \max_{1 \leq i \leq n} {X_i}  = \max (X_1,\dots,X_n) .
\end{equation}
In general, $Y_n$ and $Z_n$ are dependent random variables. For the sake of convenience, we also define the following sets    
\begin{equation}
\mathcal{D}_n \defeq \{ (y,z) \in \R_+^2 : z \leq y \leq n z \} , \quad \forall n \in \N .
\end{equation}
Subsequently, we would like to compute the CDF and PDF of the random pair $(Y_n,Z_n)$, i.e., $G_n (y,z) \defeq \Pr (Y_n \leq y,Z_n \leq z) = \int_0^z {\int_0^y  g_n(y',z')  \dif y' \dif z'}$, $\forall (y,z) \in \R_+^2$, and $g_n(y,z) = \md{}{2}{y}{}{z}{} G_n(y,z)$, $\forall (y,z) \in \mathcal{D}_n$. Note that $0 \leq Z_n \leq Y_n \leq n Z_n$, therefore $g_n(y,z) = 0$ for all $(y,z) \notin \mathcal{D}_n$.

\subsection{Recurrence Relation for the Joint CDF}

\begin{theorem} \label{theorem:Recurrence relation and continuity of joint CDF}
The joint CDF of $Y_n$ and $Z_n$, defined by \eqref{equation:sum and maximum}, is given by the following recurrence relation 
\begin{equation} \label{equation:recurrence relation of joint CDF}
\begin{split}
G_n (y,z) & = \int_0^z {G_{n-1}(y - x,z) f_n(x) \dif x}  = \int_0^{\min(y,z)} {G_{n-1}(y - x,z) f_n(x) \dif x} , \\
& \:\, \quad \forall n \in \N \setminus \{1\} , \; \forall (y,z) \in \R_+^2 ,
\end{split}
\end{equation} 
with initial condition
\begin{equation} \label{equation:initial condition for joint CDF}
G_1 (y,z) = F_1 (\min(y,z)) ,  \quad \forall (y,z) \in \R_+^2 .
\end{equation}
In addition, $G_n (y,z)$ is continuous on $\R_+^2$ for all $n \in \N$. 
\end{theorem}

\begin{proof}
Observe that $Y_n = Y_{n-1} + X_n$ and $Z_n = \max(Z_{n-1},X_n)$. Therefore, we can write 
\[
\begin{split}
G_n (y,z) & = \Pr (Y_{n-1} + X_n \leq y, \max(Z_{n-1},X_n) \leq z)   \\
& = \Pr (Y_{n-1} + X_n \leq y, Z_{n-1} \leq z, X_n \leq z)   \\
& \mathop{=}^{(a)} \int_0^\infty {\Pr (Y_{n-1} + X_n \leq y, Z_{n-1} \leq z, X_n \leq z \;|\; X_n = x) f_n(x) \dif x}   \\
& \mathop{=}^{(b)} \int_0^\infty {\Pr (Y_{n-1} + x \leq y, Z_{n-1} \leq z, x \leq z) f_n(x) \dif x}   \\
& = \int_0^z {\Pr (Y_{n-1} \leq y - x, Z_{n-1} \leq z) f_n(x) \dif x}   \\
& = \int_0^z {G_{n-1}(y - x,z) f_n(x) \dif x} = \int_0^{\min(y,z)} {G_{n-1}(y - x,z) f_n(x) \dif x}  ,  
\end{split}
\]
where equality $(a)$ is due to the law/formula of total probability, equality $(b)$ follows from the fact that $X_n$ is independent of the pair $(Y_{n-1},Z_{n-1})$, and the last equality is because $G_{n-1}(y - x,z) = 0$ when $x>y$. Regarding the initial condition,  
\[
G_1 (y,z) = \Pr (Y_1 \leq y, Z_1 \leq z) = \Pr (X_1 \leq y, X_1 \leq z) = \Pr (X_1 \leq \min(y,z)) = F_1 (\min(y,z)). 
\]

Now, it remains to show the continuity of $G_n(y,z)$ on $\R_+^2$ for all $n \in \N$. We will use mathematical induction on $n$. Basis step: For $n=1$, $G_1 (y,z) = F_1 (\min(y,z))$ is continuous on $\R_+^2$ since it is the composition of two continuous functions, $F_1(\cdot)$ on $\R_+$ and $\min(\cdot,\cdot)$ on $\R_+^2$. Inductive step: Suppose that $G_{k-1}(y,z)$ is continuous on $\R_+^2$, for some arbitrary $k \in \N \setminus \{1\}$ (inductive hypothesis). Then, according to \eqref{equation:recurrence relation of joint CDF}, $G_k (y,z) = \int_0^{\min(y,z)} {G_{k-1}(y - x,z) f_n(x) \dif x} $ is continuous on $\R_+^2$ because $G_{k-1}(y - x,z) f_n(x)$ is continuous for $0 \leq x \leq \min(y,z)$, as the product of two continuous functions, $\min(\cdot,\cdot)$ is continuous on $\R_+^2$ and the integral of a continuous function is continuous as well.   
\end{proof}

Interestingly, a similar recursive formula was given by Arendarczyk, Kozubowski and Panorska after dividing the domain of the joint CDF into several regions \cite[Eq.~(23)]{Arendarczyk-Kozubowski-Panorska_a}. Furthermore, we can derive an exact formula (a multiple integral) for the joint CDF. 

\begin{proposition} 
The joint CDF of $Y_n$ and $Z_n$ can be expressed as an $(n-1)$-dimensional integral, that is, 
\begin{equation} \label{equation:exact formula CDF}
\begin{split}
G_n (y,z) & = \bigintsss_{x_n=0}^z \cdots \bigintsss_{x_2=0}^z  G_1 \left( {y - \sum_{i=2}^n {x_i}},z \right) \prod_{i=2}^n {f_i(x_i) \dif x_i}  ,   \\
& \:\, \quad \forall n \in \N \setminus \{1\}  , \; \forall (y,z) \in \R_+^2 , 
\end{split}
\end{equation} 
where $G_1(y,z)$ is given by \eqref{equation:initial condition for joint CDF}.
\end{proposition}

\begin{proof}
Based on the recurrence relation \eqref{equation:recurrence relation of joint CDF}, we give a proof by induction on $n$. Basis step: For $n=2$, \eqref{equation:recurrence relation of joint CDF} yields $G_2 (y,z) = \int_0^z {G_1(y - x,z) f_2(x) \dif x} = \int_0^z {G_1(y - x_2,z) f_2(x_2) \dif x_2}$. Inductive step: Suppose that $G_{k-1} (y,z) = \bigintsss_{x_{k-1}=0}^z \cdots \bigintsss_{x_2=0}^z  G_1 \left( {y - \sum_{i=2}^{k-1} {x_i}},z \right) \prod_{i=2}^{k-1} {f_i(x_i) \dif x_i}$, for some arbitrary integer $k \geq 3$ (inductive hypothesis). Then, from \eqref{equation:recurrence relation of joint CDF} we obtain 
\[
\begin{split}
G_k (y,z) & = \int_0^z {G_{k-1}(y - x,z) f_k(x) \dif x}    \\
& = \bigintsss_0^z \bigintsss_{x_{k-1}=0}^z \cdots \bigintsss_{x_2=0}^z  G_1 \left( {y - x - \sum_{i=2}^{k-1} {x_i}},z \right) \left( \prod_{i=2}^{k-1} {f_i(x_i) \dif x_i} \right) {f_k(x) \dif x}   \\
& = \bigintsss_{x_k=0}^z \cdots \bigintsss_{x_2=0}^z  G_1 \left( {y - \sum_{i=2}^k {x_i}},z \right) \prod_{i=2}^k {f_i(x_i) \dif x_i}  .
\end{split}
\] 
Hence, \eqref{equation:exact formula CDF} is true for all integers $n \geq 2$. 
\end{proof}

\subsection{Recurrence Relation for the Joint PDF}

\begin{theorem} \label{theorem:Joint PDF}
The joint PDF of $Y_n$ and $Z_n$, defined by \eqref{equation:sum and maximum}, is given by the following recurrence relation 
\begin{equation} \label{equation:recurrence relation of joint PDF}
\begin{split}
g_n (y,z) & = f_n(z) \int_0^z {g_{n-1} (y - z,x) \dif x} + \int_0^z {f_n(x) g_{n-1} (y - x,z) \dif x}  \\
& = f_n(z) \int_{\tfrac{y-z}{n-1}}^{\min(y-z,z)} {g_{n-1} (y - z,x) \dif x} + \int_{\max(y-(n-1)z,0)}^{\min(y-z,z)} {f_n(x) g_{n-1} (y - x,z) \dif x}    ,  \\
& \:\, \quad  \forall n \in \N \setminus \{1\} , \; \forall (y,z) \in \mathcal{D}_n  ,
\end{split}
\end{equation}
with initial condition
\begin{equation} \label{equation:initial condition for joint PDF}
g_1 (y,z) = f_1 (y) \delta(y - z) = f_1 (z) \delta(y - z)  , \quad \forall (y,z) \in \mathcal{D}_1  .
\end{equation}
Furthermore, $g_n (y,z)$ is continuous on its domain $\mathcal{D}_n$, for all $n \in \N \setminus \{1\}$. 
\end{theorem}

\begin{proof}
For all $n \in \N$, we know that $g_n(y,z) = 0$ for all $(y,z) \notin \mathcal{D}_n$, due to the definition of $Y_n$ and $Z_n$. Hence, the determination of $g_n(y,z)$ on $\mathcal{D}_n$ is sufficient for its full definition. In addition, recall that the sum, product, and composition of finitely many continuous functions are all continuous; the integral of a continuous function is continuous as well. 

First, let us prove the initial condition \eqref{equation:initial condition for joint PDF}. Note that $G_1 (y,z)$, given by \eqref{equation:initial condition for joint CDF}, can be expressed using the Heaviside step function as follows
\[
G_1 (y,z) = F_1 (\min(y,z)) \mathop{=}^{\text{a.e.}} F_1 (y)  H(z - y) + F_1 (z)  H(y - z)   ,
\]
where the last equality holds almost everywhere on $\R_+^2$, except on the set $\{(y,z) \in \R_+^2 : \break y=z \}$ which has Lebesgue measure zero (in the geometric sense of area). As a result, 
\[ 
\begin{split}
{\dpd{}{z}{} G_1 (y,z)} & \mathop{=}^{\text{a.e.}} F_1 (y)  \delta(z - y) + f_1 (z)  H(y - z) - F_1 (z)  \delta(y - z)    \\
& = F_1 (z)  \delta(z - y) - F_1 (z)  \delta(z - y) +  f_1 (z)  H(y - z)  \\
& = f_1 (z)  H(y - z)  ,
\end{split} 
\]
and 
\[
{\dmd{}{2}{y}{}{z}{} G_1(y,z)} \mathop{=}^{\text{a.e.}} f_1 (z)  \delta(y - z) = f_1 (y) \delta(y - z)  .
\]
In the above equations, we have used the fact that $F_1(x)$ and $f_1(x)$ are bounded and piecewise continuous on $\R$ (see also Appendix~\ref{appendix:Dirac Delta Function}). Since a PDF is almost unique when it exists, meaning that any two such PDFs coincide almost everywhere, we can choose $g_1 (y,z) = f_1 (z)  \delta(y - z) = f_1 (y) \delta(y - z)$ for all $(y,z) \in \R_+^2$ or, equivalently, for all $(y,z) \in \mathcal{D}_1 = \{ (y,z) \in \R_+^2 : y = z \}$, since $g_1(y,z) = 0$ for $y \neq z$. We can also verify that 
\[
\begin{split}
\int_0^z {\int_0^y  g_1(y',z')  \dif y' \dif z'} & = \int_0^z {f_1(z') \int_0^y \delta(y' - z')  \dif y' \dif z'} = \int_0^z {f_1(z') \int_{-z'}^{y-z'} \delta(w)  \dif w \dif z'}\\
& = \int_0^z {f_1(z') [-z' \leq 0 \leq y-z'] \dif z'} = \int_0^z {f_1(z') [0 \leq z' \leq y] \dif z'}   \\ 
& =  \int_0^{\min(y,z)} {f_1(z') \dif z'} = F_1 (\min(y,z)) = G_1 (y,z)  .
\end{split}
\]

Next, we will prove \eqref{equation:recurrence relation of joint PDF} and the continuity of $g_n (y,z)$ on its domain $\mathcal{D}_n$, for all $n \in \N \setminus \{1\}$, using induction on $n$. Basis step: For $n=2$ and by exploiting \eqref{equation:recurrence relation of joint CDF}, we obtain   
\[
\begin{split}
G_2(y,z) &=  \int_0^z {G_1(y - x,z) f_2(x) \dif x} = \int_0^z {F_1 (\min(y-x,z)) f_2(x) \dif x}  \\
& = F_1(z) \int_0^{y-z} {f_2(x) \dif x} + \int_{y-z}^z {F_1(y-x) f_2(x) \dif x}   \\
& = F_1(z) F_2(y-z) + \int_{y-z}^z {F_1(y-x) f_2(x) \dif x}   .
\end{split}
\]

Due to the fact that $F_1(y-x)$ and $f_2(x)$ are continuous for $0 \leq y-z \leq x \leq z$ and by the fundamental theorem of calculus (which is applicable on this region), the partial derivative of $G_2(y,z)$ with respect to $z$ is given by
\[
\begin{split}
\dpd{}{z} G_2(y,z) & = f_1(z) F_2(y-z) - F_1(z) f_2(y-z) + F_1(y-z) f_2(z) + F_1(z) f_2(y-z)   \\
& = F_1(y-z) f_2(z) + f_1(z) F_2(y-z)   , \quad \forall (y,z) \in \mathcal{D}_2
\end{split}
\]
Subsequently, differentiation with respect to $y$ gives $g_2(y,z)={\md{}{2}{y}{}{z}{} G_2(y,z)}$, that is,
\begin{equation} \label{equation:g_2}
g_2 (y,z) =  f_1(y-z) f_2(z) + f_1(z) f_2(y-z)  , \quad  \forall (y,z) \in \mathcal{D}_2  .
\end{equation}
Note that $g_2 (y,z)$ is continuous on $\mathcal{D}_2$, since it is the sum of products of two continuous functions on $\mathcal{D}_2$, $f_1(\cdot)$ and $f_2(\cdot)$. Now, for $n=2$, the first right-hand side of \eqref{equation:recurrence relation of joint PDF} is equal to  
\[
\begin{split}
& f_2(z) \int_0^z {g_1 (y - z,x) \dif x} + \int_0^z {f_2(x) g_1(y - x,z) \dif x}   \\
= & f_2(z) \int_{\max(y-z,0)}^{\min(y-z,z)} {g_1 (y - z,x) \dif x} + \int_{\max(y-z,0)}^{\min(y-z,z)} {f_2(x) g_1 (y - x,z) \dif x}   \\
= &  f_2(z) \int_{y-z}^{y-z} {g_1 (y - z,x) \dif x} + \int_{y-z}^{y-z} {f_2(x) g_1 (y - x,z) \dif x}     \\
= &  \left( f_1(y-z) f_2(z) + f_1(z) f_2(y-z) \right) \int_{y-z}^{y-z} {\delta(y - z - x) \dif x}     \\
= &  \left( f_1(y-z) f_2(z) + f_1(z) f_2(y-z) \right) \int_0^0 {\delta(w) \dif w}     \\
= & f_1(y-z) f_2(z) + f_1(z) f_2(y-z) = g_2(y,z)   ,  \quad \forall (y,z) \in \mathcal{D}_2  .
\end{split}
\]
The first equality follows from the domain $\mathcal{D}_1$ of $g_1(y,z)$, and the second equality is because $z \leq y \leq 2z$ on $\mathcal{D}_2$. 

Inductive step: Assume that \eqref{equation:recurrence relation of joint PDF} is true for $n=k-1$ and that $g_{k-1}(y,z)$ is defined and continuous on $\mathcal{D}_{k-1} = \{ (y,z) \in \R_+^2 : z \leq y \leq (k-1)z \}$, for some arbitrary integer $k \geq 3$ (inductive hypothesis). Then by taking the second-order mixed partial derivative of \eqref{equation:recurrence relation of joint CDF}, we obtain  
\[
g_k(y,z) = {\dmd{}{2}{y}{}{z}{} G_k(y,z)} = \dmd{}{2}{y}{}{z}{} \left( \int_0^z {f_k(x) G_{k-1}(y - x,z) \dif x} \right) .
\]
In order to apply Proposition~\ref{proposition:Partial Differentiation Under the Integral Sign}, let $a(z)=0$, $b(z)=z$, and $h(x,y,z)=f_k(x) G_{k-1}(y - x,z)$. Then, because $G_{k-1} (y,z) = \int_0^z {\int_0^y  g_{k-1}(y',z')  \dif y' \dif z'}$, we get for all $(x,y,z) \in \R_+^3$
\[
\dpd{}{y} h(x,y,z) = f_k(x) \int_0^z {g_{k-1} (y - x,z') \dif z'} = f_k(x) \int_{\max(\frac{y-x}{k-1},0)}^{\min(y-x,z)} {g_{k-1} (y - x,z') \dif z'}  ,
\]

\[
\dpd{}{z} h(x,y,z) = f_k(x) \int_0^{y-x} {g_{k-1} (y',z) \dif y'} = f_k(x) \int_z^{\min(y-x,(k-1)z)} {g_{k-1} (y',z) \dif y'} ,
\]

\[
{\dmd{}{2}{y}{}{z}{} h(x,y,z)} = f_k(x) g_{k-1} (y - x,z) .
\]
The limits of integration were appropriately modified by exploiting the domain $\mathcal{D}_{k-1}$ of $g_{k-1} (y,z)$ in order to remove zero values of the integrands. 

There is a technical point here regarding the assumptions of Proposition~\ref{proposition:Partial Differentiation Under the Integral Sign}. Observe that: i) $h(x,y,z)$ is continuous on $\{(x,y,z) \in \R_+^3 : x \leq y \}$; note that $G_{k-1}(y,z)$ is continuous on $\R_+^2$ due to Theorem~\ref{theorem:Recurrence relation and continuity of joint CDF}, ii) $\pd{}{y} h(x,y,z)$ exists (and is continuous) on $\{(x,y,z) \in \R_+^3 : \min(y-x,z) \geq \max(\frac{y-x}{k-1},0) \} = \{(x,y,z) \in \R_+^3 : x \leq y \leq x+(k-1)z \}$; note that $\min(\cdot,\cdot)$ and $\max(\cdot,\cdot)$ are both continuous functions, and $\min(a,b) \geq \max(c,d) \iff (a \geq c) \land (a \geq d) \land (b \geq c) \land (b \geq d)$, iii) $\pd{}{z} h(x,y,z)$ exists and is continuous on $\{(x,y,z) \in \R_+^3 : \min(y-x,(k-1)z) \geq z \} = \{(x,y,z) \in \R_+^3 : x+z \leq y \}$, and iv) $\md{}{2}{y}{}{z}{} h(x,y,z)$ exists and is continuous on $\{(x,y,z) \in \R_+^3 : x+z \leq y \leq x+(k-1)z \}$. The intersection of all these sets is equal to $\{(x,y,z) \in \R_+^3 : x+z \leq y \leq x+(k-1)z \}$. Since the domain of integration of $h(x,y,z)$ is $0 \leq x \leq z$, Proposition~\ref{proposition:Partial Differentiation Under the Integral Sign} is applicable on $\{(x,y,z) \in \R_+^3 : x+z \leq y \leq x+(k-1)z ,\; x \leq z\}$, thus yielding 
\[
\begin{split}
g_k(y,z) & = f_k(z) \int_0^z {g_{k-1} (y - z,z') \dif z'} + \int_0^z {f_k(x) g_{k-1} (y - x,z) \dif x}  \\
& = f_k(z) \int_{\max(\frac{y-z}{k-1},0)}^{\min(y-z,z)} {g_{k-1} (y - z,z') \dif z'} +  \int_{\max(y-(k-1)z,0)}^{\min(y-z,z)} {f_k(x) g_{k-1} (y - x,z) \dif x}   \\
& = f_k(z) \int_{\tfrac{y-z}{k-1}}^{\min(y-z,z)} {g_{k-1} (y - z,z') \dif z'} + \int_{\max(y-(k-1)z,0)}^{\min(y-z,z)} {f_k(x) g_{k-1} (y - x,z) \dif x}    ,
\end{split}
\]
for all $(y,z) \in \R_+^2$ such that $z \leq y \leq kz$, that is, for all $(y,z) \in \mathcal{D}_k$. From the last equality, it is clear that $g_k(y,z)$ is continuous on $\mathcal{D}_k$, because the integrands are continuous on their respective domain of integration, and the limits of integration are continuous as well.

Therefore, \eqref{equation:recurrence relation of joint PDF} is true and $g_n (y,z)$ is continuous on its domain $\mathcal{D}_n$, for all integers $n \geq 2$. This completes the proof. 
\end{proof}

\begin{remark}
An alternative proof of \eqref{equation:recurrence relation of joint PDF} for $n=2$ is based on Proposition~\ref{proposition:Partial Differentiation Under the Integral Sign} together with Remark~\ref{remark:Uniform convergence assumption}. In particular, let $a(z)=0$, $b(z)=z$, and $h(x,y,z)=f_2(x) G_1(y - x,z)$. Since $G_1(y,z) = \int_0^z {\int_0^y  g_1(y',z')  \dif y' \dif z'}$, we obtain for all $(x,y,z) \in \R_+^3$  
\[
\begin{split}
\dpd{}{y} h(x,y,z) & = f_2(x) \int_0^z {g_1 (y - x,z') \dif z'} = f_2(x) \int_{\max(y-x,0)}^{\min(y-x,z)} {g_1 (y - x,z') \dif z'}  \\
& = f_1(y-x) f_2(x) \int_{\max(y-x,0)}^{\min(y-x,z)} {\delta(y-x-z')  \dif z'}  \\
& = f_1(y-x) f_2(x) \int_{\max(y-x-z,0)}^{\min(y-x,0)} {\delta(w)  \dif w}   \\
& = f_1(y-x) f_2(x) [\max(y-x-z,0) \leq 0 \leq \min(y-x,0)]   \\
& = f_1(y-x) f_2(x) [x \leq y \leq x+z]   ,
\end{split}
\]

\[
\begin{split}
\dpd{}{z} h(x,y,z) & = f_2(x) \int_0^{y-x} {g_1 (y',z) \dif y'} = f_2(x) \int_z^{\min(y-x,z)} {g_1 (y',z) \dif y'} \\ & = f_1(z) f_2(x) \int_z^{\min(y-x,z)} {\delta(y'-z) \dif y'}  =  f_1(z) f_2(x) \int_0^{\min(y-x-z,0)} {\delta(w) \dif w}    \\
& = f_1(z) f_2(x) [0 \leq \min(y-x-z,0)] = f_1(z) f_2(x) [x+z \leq y]   ,
\end{split}
\]

\[
{\dmd{}{2}{y}{}{z}{} h(x,y,z)} = f_2(x) g_1 (y - x,z) = f_1(z) f_2(x) \delta(y-x-z) = f_1(z) f_2(y-z) \delta(y-z-x) ,
\]

\[
\begin{split}
\int_0^z {\dmd{}{2}{y}{}{z}{} h(x,y,z) \dif x} & = \int_0^z {f_2(x) g_1 (y - x,z) \dif x} = \int_{\max(y-z,0)}^{\min(y-z,z)} {f_2(x) g_1 (y - x,z) \dif x}    \\
& = f_1(z) f_2(y-z) \int_{\max(y-z,0)}^{\min(y-z,z)} {\delta(y-z-x) \dif x}  =    \\ 
& = f_1(z) f_2(y-z) \int_{\max(y-2z,0)}^{\min(y-z,0)} {\delta(w) \dif w}   \\
& = f_1(z) f_2(y-z) [\max(y-2z,0) \leq 0 \leq \min(y-z,0)]   \\
& =  f_1(z) f_2(y-z) [z \leq y \leq 2z]   .
\end{split}
\]

Now, we can observe that: i) $h(x,y,z)$ is continuous on $\{(x,y,z) \in \R_+^3 : x \leq y \}$; note that $G_1(y,z)$ is continuous on $\R_+^2$ because of Theorem~\ref{theorem:Recurrence relation and continuity of joint CDF}, ii) $\pd{}{y} h(x,y,z)$ exists (and is continuous) on $\{(x,y,z) \in \R_+^3 : x \leq y \leq x+z \}$, and iii) $\pd{}{z} h(x,y,z)$ exists and is continuous on $\{(x,y,z) \in \R_+^3 : x+z \leq y \}$. The intersection of these sets is equal to $\{(x,y,z) \in \R_+^3 : x+z \leq y \leq x+z \} = \{(x,y,z) \in \R_+^3 : y = x+z \}$. Because the domain of integration of $h(x,y,z)$ is $0 \leq x \leq z$, the assumptions of Proposition~\ref{proposition:Partial Differentiation Under the Integral Sign} regarding $h(x,y,z)$, $\pd{}{y} h(x,y,z)$, and $\pd{}{z} h(x,y,z)$ are all satisfied on $\{(x,y,z) \in \R_+^3 : y = x+z ,\; x \leq z\}$. Furthermore, the integral $\int_0^z {\md{}{2}{y}{}{z}{} h(x,y,z) \dif x}$ converges to $f_1(z) f_2(y-z)$ uniformly on $\{(y,z) \in \R_+^2 : z \leq y \leq 2z \} = \mathcal{D}_2$; see also Appendix~\ref{appendix:Dirac Delta Function} for more details on this topic. Consequently, based on Proposition~\ref{proposition:Partial Differentiation Under the Integral Sign} and Remark~\ref{remark:Uniform convergence assumption}, we conclude that 
\[
\begin{split}
g_2(y,z) & = {\dmd{}{2}{y}{}{z}{} G_2(y,z)} = \dmd{}{2}{y}{}{z}{} \left( \int_0^z {f_2(x) G_1(y - x,z) \dif x} \right)   \\
& = f_2(z) \int_0^z {g_1 (y - z,z') \dif z'} + \int_0^z {f_2(x) g_1 (y - x,z) \dif x}    \\
& = f_2(z) \int_{y-z}^{\min(y-z,z)} {g_1 (y - z,z') \dif z'} +  \int_{\max(y-z,0)}^{\min(y-z,z)} {f_2(x) g_1 (y - x,z) \dif x}   \\
& = f_1(y-z) f_2(z) + f_1(z) f_2(y-z) ,  \quad \forall (y,z) \in \mathcal{D}_2  .
\end{split}
\] 
\end{remark}

\begin{remark}
From \eqref{equation:g_2}, we can observe that the joint PDF $g_2 (y,z)$ remains the same by interchanging $X_1$ and $X_2$, and therefore $f_1(\cdot)$ and $f_2(\cdot)$, as it was expected. Moreover, in the case of exponential random variables with $f_i(x) = \lambda_i e^{-\lambda_i x}$, $x \geq 0$ and $\lambda_i > 0$, we obtain $g_2(y,z) = \lambda_1 \lambda_2 (e^{-\lambda_1 (y-z) - \lambda_2 z} + e^{-\lambda_1 z - \lambda_2 (y-z)})$, $\forall (y,z) \in \mathcal{D}_2$. This is consistent with \cite[Eq.~(7)]{Arendarczyk-Kozubowski-Panorska_a}. 
\end{remark}

\subsubsection{\textnormal{\textbf{Independent and Identically Distributed Random Variables.}}}

The following lemma will be useful for simplifying \eqref{equation:recurrence relation of joint PDF} in the i.i.d.~case. 

\begin{lemma} \label{lemma:Integration property of joint PDF}
For all $n \in \N \setminus \{1\}$ and for all $(y,z) \in \mathcal{D}_{n+1}$, it holds that  
\[
\int_0^z  {g_n(y-z,x) \dif x}  =  \int_0^z {f_n(x) \int_0^z {g_{n-1}(y-z-x,w) \dif w} \dif x} .
\]
\end{lemma}

\begin{proof}
By virtue of \eqref{equation:recurrence relation of joint PDF}, we have  
\[
g_n(y-z,x) = f_n(x) \int_0^x {g_{n-1}(y-z-x,w) \dif w} + \int_0^x {f_n(w) g_{n-1}(y-z-w,x)  \dif w}  ,
\]
on the set $\{(x,y,z) \in \R^3 : x \geq 0, \; y \geq z, \; x \leq y-z \leq nx\} = \{(x,y,z) \in \R^3 : x \geq 0, \; z+x \leq y \leq z+nx\}$. Then, by integrating over the region $0 \leq x \leq z$, we obtain 
\[
\scalebox{0.95}{$
\begin{split}
\int_{0}^z  {g_n(y-z,x) \dif x}  & =  \int_0^z {f_n(x) \int_0^x {g_{n-1}(y-z-x,w) \dif w}  \dif x}  +  \int_0^z {\int_0^x {f_n(w) g_{n-1}(y-z-w,x)  \dif w}  \dif x}  \\
& \mathop{=}^{(c)} \int_0^z {f_n(x) \int_0^x {g_{n-1}(y-z-x,w) \dif w}  \dif x}  +  \int_0^z {f_n(w) \int_w^z {g_{n-1}(y-z-w,x)  \dif x}  \dif w}   \\ 
& \mathop{=}^{(d)}  \int_0^z {f_n(x) \int_0^x {g_{n-1}(y-z-x,w) \dif w}  \dif x}  +  \int_0^z {f_n(x) \int_x^z {g_{n-1}(y-z-x,w)  \dif w}  \dif x}       \\ 
& = \int_0^z {f_n(x) \left( \int_0^x {g_{n-1}(y-z-x,w) \dif w} + \int_x^z {g_{n-1}(y-z-x,w)  \dif w} \right) \dif x }     \\ 
& =  \int_0^z {f_n(x) \int_0^z {g_{n-1}(y-z-x,w) \dif w} \dif x}   ,   
\end{split}  $} 
\]
for all $(y,z) \in \R_+^2$ such that $z \leq y \leq (n+1)z$, that is, for all $(y,z) \in \mathcal{D}_{n+1}$. Note that equalities $(c)$ and $(d)$ have been derived by interchanging the order of integration \linebreak ($\{0 \leq x \leq z\} \land \{0 \leq w \leq x\} \iff \{0 \leq w \leq z\} \land \{w \leq x \leq z\}$) and by interchanging the variables $w$, $x$ in the second double integral, respectively.  
\end{proof}

\begin{theorem} \label{theorem:recurrence relation of joint PDF_iid}
Let $f(x)$ be the common PDF of random variables $\{X_i\}_{i \in \N}$, that is, \linebreak $f_i(x) = f(x)$, $\forall i \in \N$. Then, the PDF of $(Y_n,Z_n)$ is given by the recurrence relation 
\begin{equation} \label{equation:recurrence relation of joint PDF_iid}
\begin{split}
g_n (y,z) & = n f(z) \int_0^z {g_{n-1} (y - z,x) \dif x} = n f(z) \int_{\tfrac{y-z}{n-1}}^{\min(y-z,z)} {g_{n-1} (y - z,x) \dif x} , \\ 
& \:\, \quad \forall n \in \N \setminus \{1\}, \; \forall (y,z) \in \mathcal{D}_n ,  
\end{split}
\end{equation}
with initial condition 
\begin{equation} \label{equation:initial condition for joint PDF_iid}
g_1(y,z) = f(y) \delta(y - z) = f(z) \delta(y - z), \quad \forall (y,z) \in \mathcal{D}_1  .
\end{equation}
\end{theorem}
 
\begin{proof}
The initial condition can be deduced from \eqref{equation:initial condition for joint PDF}. Observe that the second equality in \eqref{equation:recurrence relation of joint PDF_iid} follows directly from the first by taking into account the domain $\mathcal{D}_{n-1}$ of $g_{n-1}(y,z)$ and that $y \geq z$ on $\mathcal{D}_n$. Therefore, it is sufficient to prove the first equality, using induction on $n$. 

Basis step: For $n=2$, the right-hand side of the first equality in \eqref{equation:recurrence relation of joint PDF_iid} gives
\[
2 f(z) \int_0^z {g_1 (y - z,x) \dif x} =  2 f(y-z) f(z) [z \leq y \leq 2z] = 2 f(y-z) f(z)  ,\quad \forall (y,z) \in \mathcal{D}_2  .
\]
In addition, by setting $f_1(x) = f_2(x) = f(x)$ in \eqref{equation:g_2}, we conclude that
\[
g_2(y,z) = 2 f(y-z) f(z) ,  \quad  \forall (y,z) \in \mathcal{D}_2 .
\]
Thus, $g_2(y,z) = 2 f(z) \int_0^z {g_1 (y - z,x) \dif x}$, for all $(y,z) \in \mathcal{D}_2$.

Inductive step: Suppose that 
\[
g_{k-1} (y,z) = (k-1) f(z) \int_0^z {g_{k-2} (y - z,w) \dif w} , \quad \forall (y,z) \in \mathcal{D}_{k-1} ,
\]
for some arbitrary integer $k \geq 3$ (inductive hypothesis). Then, \eqref{equation:recurrence relation of joint PDF} for $n=k$ and Lemma~\ref{lemma:Integration property of joint PDF} for $n=k-1$ yield  
\[
\begin{split}
g_k(y,z) & = f(z) \int_0^z {g_{k-1} (y - z,x) \dif x} + \int_0^z {f(x) g_{k-1} (y - x,z) \dif x}  \\
& = f(z) \int_0^z {g_{k-1} (y - z,x) \dif x} + (k-1) f(z) \int_0^z {f(x) \int_0^z {g_{k-2}(y-x-z,w) \dif w} \dif x}    \\
& =  f(z) \int_0^z {g_{k-1} (y - z,x) \dif x} + (k-1) f(z) \int_0^z {g_{k-1} (y - z,x) \dif x}    \\
& =  k f(z) \int_0^z {g_{k-1} (y - z,x) \dif x}   , \quad \forall (y,z) \in \mathcal{D}_k .    
\end{split}  
\] 
Note that the second equality holds on $\mathcal{D}_k$, because we have $g_{k-1} (y-x,z) = (k-1) f(z) \break \int_0^z {g_{k-2} (y - x - z,w) \dif w}$ on $\{(x,y,z) \in \R^3 : z \geq 0, \; y \geq x, \; z \leq y-x \leq (k-1)z\} = \{(x,y,z) \in \R^3 : z \geq 0, \; x+z \leq y \leq x+(k-1)z\}$ and the domain of integration is $0 \leq x \leq z$. 
\end{proof}

\begin{remark}
Eq.~\eqref{equation:recurrence relation of joint PDF_iid} is in complete agreement with \cite[Eq.~(2.18)]{Qeadan-Kozubowski-Panorska}, which was derived using transformations of order statistics. 
\end{remark}

Moreover, we can prove an exact formula (a multiple integral) for the joint PDF in the i.i.d.~case as follows.

\begin{proposition}
Suppose that all random variables $X_i$ have the same PDF $f(x)$, that is, $f_i(x) = f(x)$, $\forall i \in \N$. Then, the PDF of $(Y_n,Z_n)$ can be expressed as an $(n-1)$-dimensional integral as follows
\[
g_n (y,z) =  {n!} f(z) \bigintsss_{x_n=0}^z {f(x_n) \bigintsss_{x_{n-1}=0}^{x_n} f(x_{n-1}) \cdots \bigintsss_{x_3=0}^{x_4} f(x_3) \bigintsss_{x_2=0}^{x_3} g_1\left( {y - z - \sum_{i=3}^n {x_i}},x_2 \right) \prod_{i=2}^n {\dif x_i}} ,
\] 
for all $n \in \N \setminus \{1\}$ and for all $(y,z) \in \mathcal{D}_n$, where $g_1(y,z)$ is given by \eqref{equation:initial condition for joint PDF_iid}.
\end{proposition}

\begin{proof}
Once again, we apply induction on $n$. Basis step: For $n=2$, \eqref{equation:recurrence relation of joint PDF_iid} yields $g_2 (y,z) = 2 f(z) \int_0^z {g_1 (y - z,x) \dif x} = 2! f(z) \int_0^z {g_1 (y - z,x_2) \dif x_2}$. Inductive step: Suppose that 
\[
g_{k-1} (y,z) =  {(k-1)!} f(z) \bigintsss_{x_{k-1}=0}^z {f(x_{k-1}) \cdots \bigintsss_{x_3=0}^{x_4} f(x_3) \bigintsss_{x_2=0}^{x_3} g_1\left( {y - z - \sum_{i=3}^{k-1} {x_i}},x_2 \right) \prod_{i=2}^{k-1} {\dif x_i}} ,
\]
for some arbitrary integer $k \geq 3$ (inductive hypothesis). Then, from \eqref{equation:recurrence relation of joint PDF_iid} we obtain 
\[
\scalebox{0.92}{$
\begin{split}
g_k (y,z) & = k f(z) \int_0^z {g_{k-1} (y - z,x) \dif x}  \\ 
& = k! f(z) \bigintsss_0^z {f(x) \bigintsss_{x_{k-1}=0}^{x} f(x_{k-1}) \cdots \bigintsss_{x_3=0}^{x_4} f(x_3) \bigintsss_{x_2=0}^{x_3} g_1\left( {y - z - x - \sum_{i=3}^{k-1} {x_i}},x_2 \right) \left( \prod_{i=2}^{k-1} {\dif x_i} \right) \dif x}     \\
& = k! f(z) \bigintsss_{x_k=0}^z {f(x_k) \bigintsss_{x_{k-1}=0}^{x_k} f(x_{k-1}) \cdots \bigintsss_{x_3=0}^{x_4} f(x_3) \bigintsss_{x_2=0}^{x_3} g_1\left( {y - z - \sum_{i=3}^k {x_i}},x_2 \right) \prod_{i=2}^k {\dif x_i}}  .
\end{split} $}
\] 
This concludes the proof. 
\end{proof}

\subsection{Application:~Peak-to-Average Ratio}

The peak-to-average ratio of the $n^\text{th}$ order is defined by   
\[
\rho_n \defeq \frac{\max_{1 \leq i \leq n} X_i}{\tfrac{1}{n} \sum_{i=1}^n {X_i}} = \frac{Z_n}{\tfrac{1}{n} Y_n} , \quad \text{whenever}\ Y_n \neq 0 .
\] 
The probability that $\rho_n$ lies within the interval $[\alpha,\beta]$, $0 \leq \alpha \leq \beta$, is given by 
\begin{equation}
\begin{split}
\Pr (\alpha \leq \rho_n \leq \beta) & = \Pr (Y_n > 0, \tfrac{\alpha}{n} Y_n \leq Z_n \leq \tfrac{\beta}{n} Y_n) = \int_0^\infty {\int_{\tfrac{\alpha}{n} y}^{\tfrac{\beta}{n} y}  g_n(y,z)  \dif z \dif y} \\
& =  \int_0^\infty {\int_{\max\left( \tfrac{\alpha}{n} y , \tfrac{1}{n} y \right)}^{\min\left( \tfrac{\beta}{n} y , y \right)}  g_n(y,z)  \dif z \dif y} = \int_0^\infty {\int_{\tfrac{\max(\alpha,1)}{n} y}^{\min\left( \tfrac{\beta}{n},1 \right) y}  g_n(y,z)  \dif z \dif y}, 
\end{split} 
\end{equation}
where $g_n(y,z)$ is computed recursively according to Theorem~\ref{theorem:Joint PDF}. 

For example, this kind of calculation can be useful in the performance analysis of communication systems \cite{Morrison-Tobias}. In such systems the peak-to-average ratio is also known as ``peak-to-average power ratio (PAPR)'' because each $X_i$ represents the power of a random signal, which is definitely a nonnegative quantity (i.e., $X_i \geq 0$).

\subsection{Other Applications}

By exploiting the recursive formula of the joint PDF in Theorem~\ref{theorem:Joint PDF}, we can calculate several probabilistic quantities. Some indicative examples are the marginal PDFs 
\[f_{Y_n}(y) = \int_0^\infty {g_n(y,z) \dif z} \quad \text{and} \quad f_{Z_n}(z) = \int_0^\infty {g_n(y,z) \dif y},\] 
the conditional PDFs 
\[f_{Y_n|Z_n} (y|z) = \frac{g_n(y,z)}{f_{Z_n}(z)} = \frac{g_n(y,z)}{\int_0^\infty {g_n(y',z) \dif y'}} \quad \text{and} \quad f_{Z_n|Y_n} (z|y) = \frac{g_n(y,z)}{f_{Y_n}(y)} = \frac{g_n(y,z)}{\int_0^\infty {g_n(y,z') \dif z'}},\] 
as well as the expectation 
\[\mathbb{E} (h(Y_n,Z_n)) = \int_0^\infty {\int_0^\infty  h(y,z) g_n(y,z)  \dif y \dif z} = \int_0^\infty {\int_z^{nz}  h(y,z) g_n(y,z)  \dif y \dif z},\] 
where $h(y,z)$ is a real-valued function defined on $\R_+^2$ such that the last double integral exists. Note that the joint moment $\mathbb{E} (Y_n^\alpha Z_n^\beta)$, with $\alpha, \beta \in \R$, can be obtained from the latter expression by setting $h(y,z) = y^\alpha z^\beta$, whenever the double integral is defined.

\section{Discrete Random Variables}
\label{section:Discrete Random Variables}

In this section, we extend the methodology to the case of discrete random variables. Let $\{\widehat{X}_i\}_{i \in \N}$ be a sequence of \emph{independent, not necessarily identically distributed,} nonnegative discrete random variables. For simplicity, we assume that all of them are \emph{integer-valued}, i.e., $\widehat{X}_i \in \N_0$, $\forall i \in \N$. In addition, let $\widehat{f}_i(k) \defeq \Pr (\widehat{X}_i = k) = \widehat{F}_i (k) - \widehat{F}_i (k-1),\ \forall k \in \N_0$, and $\widehat{F}_i(x) \defeq \Pr (\widehat{X}_i \leq x) = \sum_{k'=0}^{\lfloor x \rfloor} {\widehat{f}_i(k')},\ \forall x \geq 0$, be the PMF and CDF of $\widehat{X}_i$, respectively. As in the continuous case, we define the sum and maximum of the first $n \in \N$ variables     
\begin{equation} \label{equation:sum and maximum_discrete}
\widehat{Y}_n \defeq \sum_{i=1}^n {\widehat{X}_i}= \widehat{X}_1 + \cdots + \widehat{X}_n , \qquad  \widehat{Z}_n \defeq \max_{1 \leq i \leq n} {\widehat{X}_i}  = \max (\widehat{X}_1,\dots,\widehat{X}_n) .
\end{equation}
Generally, $\widehat{Y}_n$ and $\widehat{Z}_n$ are dependent random variables. Moreover, it will be helpful to define the following sets   
\begin{equation}
\widehat{\mathcal{D}}_n \defeq \{ (l,m) \in \N_0^2 : m \leq l \leq n m \} , \quad \forall n \in \N .
\end{equation}
Afterwards, we will study the joint CDF and PMF of $\widehat{Y}_n$ and $\widehat{Z}_n$, that is, $\widehat{G}_n (y,z) \defeq \Pr (\widehat{Y}_n \leq y,\widehat{Z}_n \leq z ) = \sum_{l=0}^{\lfloor y \rfloor} \sum_{m=0}^{\lfloor z \rfloor} {\widehat{g}_n (l,m)},\ \forall (y,z) \in \R_+^2$, and $\widehat{g}_n (l,m) \defeq \Pr (\widehat{Y}_n = l,\widehat{Z}_n = m),\ \forall (l,m) \in \widehat{\mathcal{D}}_n$. Observe that $(\widehat{Y}_n, \widehat{Z}_n) \in \N_0^2$ and $\widehat{Z}_n \leq \widehat{Y}_n \leq n \widehat{Z}_n$, thus $\widehat{g}_n (l,m) = 0$ for all $(l,m) \notin \widehat{\mathcal{D}}_n$. In particular, we will derive the discrete counterparts of continuous-case recurrence relations, but with a significant difference between the joint PDF and PMF (since the Leibniz integral rule is no longer applicable in the discrete case).

\subsection{Recurrence Relation for the Joint CDF}

\begin{theorem} \label{theorem:Recurrence relation of joint CDF_discrete}
The joint CDF of $\widehat{Y}_n$ and $\widehat{Z}_n$, defined by \eqref{equation:sum and maximum_discrete}, is given by the following recurrence relation 
\begin{equation} \label{equation:recurrence relation of joint CDF_discrete}
\begin{split}
\widehat{G}_n (y,z) & = \sum_{k=0}^{\lfloor z \rfloor} {\widehat{G}_{n-1}(y - k,z) \widehat{f}_n(k)} = \sum_{k=0}^{\min(\lfloor y \rfloor,\lfloor z \rfloor)} {\widehat{G}_{n-1}(y - k,z) \widehat{f}_n(k)} ,  \\
& \:\, \quad \forall n \in \N \setminus \{1\} , \;  \forall (y,z) \in \R_+^2  ,
\end{split}
\end{equation} 
with initial condition
\begin{equation} \label{equation:initial condition for joint CDF_discrete}
\widehat{G}_1 (y,z) = \widehat{F}_1 (\min(y,z)) ,  \quad \forall (y,z) \in \R_+^2 .
\end{equation} 
\end{theorem}

\begin{proof}
In a similar way, using the law of total probability, we can write
\[
\begin{split}
\widehat{G}_n (y,z) & = \Pr (\widehat{Y}_{n-1} + \widehat{X}_n \leq y, \max(\widehat{Z}_{n-1},\widehat{X}_n) \leq z)   \\
& = \Pr (\widehat{Y}_{n-1} + \widehat{X}_n \leq y, \widehat{Z}_{n-1} \leq z, \widehat{X}_n \leq z)  \\
& = \sum_{k=0}^{\infty} {\Pr (\widehat{Y}_{n-1} + \widehat{X}_n \leq y, \widehat{Z}_{n-1} \leq z, \widehat{X}_n \leq z \;|\; \widehat{X}_n = k) \Pr (\widehat{X}_n = k)}  \\
& = \sum_{k=0}^{\lfloor z \rfloor} {\widehat{G}_{n-1}(y - k,z) \widehat{f}_n(k)}  =  \sum_{k=0}^{\min(\lfloor y \rfloor,\lfloor z \rfloor)} {\widehat{G}_{n-1}(y - k,z) \widehat{f}_n(k)}  ,
\end{split}
\]
where the last equality is because $\widehat{G}_{n-1}(y - k,z) = 0$ if $k>y$. The initial condition is   
\[ 
\widehat{G}_1 (y,z) = \Pr (\widehat{Y}_1 \leq y, \widehat{Z}_1 \leq z) = \Pr (\widehat{X}_1 \leq \min(y,z)) = \widehat{F}_1 (\min(y,z)) .   \qedhere
\]   
\end{proof}

Note that $\widehat{G}_n (y,z)$ is \emph{not} continuous on $\R_+^2$, in contrast to $G_n (y,z)$ in the case of continuous random variables (according to Theorem~\ref{theorem:Recurrence relation and continuity of joint CDF}). Moreover, we can derive an exact formula (a~multiple~sum) for the joint CDF. 

\begin{proposition} 
The joint CDF of $\widehat{Y}_n$ and $\widehat{Z}_n$ can be expressed as an $(n-1)$-dimensional sum, that is,
\begin{equation} \label{equation:exact formula CDF_discrete}
\widehat{G}_n (y,z) = \mathlarger{\mathlarger{\sum}}_{k_n=0}^{\lfloor z \rfloor} \cdots \mathlarger{\mathlarger{\sum}}_{k_2=0}^{\lfloor z \rfloor}  \widehat{G}_1 \left( {y - \sum_{i=2}^n {k_i}},z \right) \prod_{i=2}^n {\widehat{f}_i(k_i)}  , \quad \forall n \in \N \setminus \{1\}  , \; \forall (y,z) \in \R_+^2 , 
\end{equation}
where $\widehat{G}_1(y,z)$ is given by \eqref{equation:initial condition for joint CDF_discrete}. 
\end{proposition}

\begin{proof}
Based on the recurrence relation \eqref{equation:recurrence relation of joint CDF_discrete}, we provide a proof by induction on $n$. Basis step: For $n=2$, \eqref{equation:recurrence relation of joint CDF_discrete} gives $\widehat{G}_2 (y,z) = \sum_{k=0}^{\lfloor z \rfloor} {\widehat{G}_1(y - k,z) \widehat{f}_2(k)} = \sum_{k_2=0}^{\lfloor z \rfloor} {\widehat{G}_1(y - k_2,z) \widehat{f}_2(k_2)}$. Inductive step: Suppose that $\widehat{G}_{\tau-1} (y,z) = \sum_{k_{\tau-1}=0}^{\lfloor z \rfloor} \cdots \sum_{k_2=0}^{\lfloor z \rfloor}  \widehat{G}_1 \left( {y - \sum_{i=2}^{\tau-1} {k_i}},z \right) \prod_{i=2}^{\tau-1} {\widehat{f}_i(k_i)}$, for some arbitrary integer $\tau \geq 3$ (inductive hypothesis). Then, from \eqref{equation:recurrence relation of joint CDF_discrete} we get 
\[
\begin{split}
\widehat{G}_\tau (y,z) & = \sum_{k=0}^{\lfloor z \rfloor} {\widehat{G}_{\tau-1}(y - k,z) \widehat{f}_\tau(k)}   \\
& = \mathlarger{\mathlarger{\sum}}_{k=0}^{\lfloor z \rfloor} \mathlarger{\mathlarger{\sum}}_{k_{\tau-1}=0}^{\lfloor z \rfloor} \cdots \mathlarger{\mathlarger{\sum}}_{k_2=0}^{\lfloor z \rfloor}  \widehat{G}_1 \left( {y - k - \sum_{i=2}^{\tau-1} {k_i}},z \right) \left( \prod_{i=2}^{\tau-1} {\widehat{f}_i(k_i)} \right) {\widehat{f}_\tau (k)}   \\
& = \mathlarger{\mathlarger{\sum}}_{k_\tau=0}^{\lfloor z \rfloor} \cdots \mathlarger{\mathlarger{\sum}}_{k_2=0}^{\lfloor z \rfloor}  \widehat{G}_1 \left( {y - \sum_{i=2}^\tau {k_i}},z \right) \prod_{i=2}^\tau {\widehat{f}_i(k_i)}  .
\end{split}
\] 
Therefore, \eqref{equation:exact formula CDF_discrete} is true for all integers $n \geq 2$. 
\end{proof}

\subsection{Recurrence Relation for the Joint PMF}

\begin{theorem} \label{theorem:Joint PMF}
The joint PMF of $\widehat{Y}_n$ and $\widehat{Z}_n$, defined by \eqref{equation:sum and maximum_discrete}, is given by the following recurrence relation 
\begin{equation} \label{equation:recurrence relation of joint PMF}
\begin{split}
\widehat{g}_n (l,m) & = \widehat{f}_n(m) \sum_{k=0}^m { \widehat{g}_{n-1} (l-m,k) }  +  \sum_{k=0}^{m-1} {\widehat{f}_n(k) \widehat{g}_{n-1} (l-k,m)}  \\
& = \widehat{f}_n(m) \sum_{k=\left\lceil \tfrac{l-m}{n-1} \right\rceil}^{\min(l-m,m)} { \widehat{g}_{n-1} (l-m,k) }  +  \sum_{k=\max\left( l-(n-1)m,0 \right)}^{\min(l-m,m-1)} {\widehat{f}_n(k) \widehat{g}_{n-1} (l-k,m)} ,   \\
& \:\, \quad  \forall n \in \N \setminus \{1\} , \; \forall (l,m) \in \widehat{\mathcal{D}}_n  ,
\end{split}
\end{equation}
with initial condition
\begin{equation} \label{equation:initial condition for joint PMF}
\widehat{g}_1 (l,m) = \widehat{f}_1(l) = \widehat{f}_1(m)  , \quad \forall (l,m) \in \widehat{\mathcal{D}}_1  .
\end{equation} 
\end{theorem}

\begin{proof}
For all $n \in \N$, we know that $\widehat{g}_n (l,m) = 0$ for all $(l,m) \notin \widehat{\mathcal{D}}_n$, because of the definition of $\widehat{Y}_n$ and $\widehat{Z}_n$. First, let us prove the initial condition \eqref{equation:initial condition for joint PMF}: 
\[
\begin{split}
\widehat{g}_1 (l,m) & = \Pr (\widehat{Y}_1 = l,\widehat{Z}_1 = m) = \Pr (\widehat{X}_1 = l,\widehat{X}_1 = m)  \\
& = \widehat{f}_1(l) [l = m]  = \widehat{f}_1(m) [l = m] , \quad \forall (l,m) \in \N_0^2 , 
\end{split}
\]
or, equivalently, $\widehat{g}_1 (l,m) = \widehat{f}_1(l) = \widehat{f}_1(m)$, $\forall (l,m) \in \widehat{\mathcal{D}}_1 = \{ (l,m) \in \N_0^2 : l = m \}$.

In the discrete case the recursive application of Proposition~\ref{proposition:Partial Differentiation Under the Integral Sign} does not work, so we need an alternative approach to prove \eqref{equation:recurrence relation of joint PMF}. In particular, we make use of the law of total probability as follows
\[
\begin{split}
\widehat{g}_n (l,m) \defeq & \Pr (\widehat{Y}_n = l,\widehat{Z}_n = m) = \Pr (\widehat{Y}_{n-1} + \widehat{X}_n = l, \max(\widehat{Z}_{n-1},\widehat{X}_n) = m)   \\
 = & \sum_{k=0}^{\infty} {\Pr (\widehat{Y}_{n-1} + \widehat{X}_n = l, \max(\widehat{Z}_{n-1},\widehat{X}_n) = m \;|\; \widehat{X}_n = k) \Pr (\widehat{X}_n = k)}     \\
 = &  \sum_{k=0}^{\infty} {\widehat{f}_n (k) \Pr (\widehat{Y}_{n-1} = l-k, \max(\widehat{Z}_{n-1},k) = m)}    \\
 = &  \sum_{k=0}^{\infty} {\widehat{f}_n (k) \left( \Pr (\widehat{Y}_{n-1} = l-k, \max(\widehat{Z}_{n-1},k) = m, k=m) \right. }  \\
&  \,  {\left. + \Pr (\widehat{Y}_{n-1} = l-k, \max(\widehat{Z}_{n-1},k) = m, k \leq m-1) \right)}    \\
 = & {\widehat{f}_n (m) \Pr (\widehat{Y}_{n-1} = l-m, \max(\widehat{Z}_{n-1},m) = m)}  +  \sum_{k=0}^{m-1} {\widehat{f}_n (k) \Pr (\widehat{Y}_{n-1} = l-k, \widehat{Z}_{n-1} = m)}   \\
 = & {\widehat{f}_n (m) \Pr (\widehat{Y}_{n-1} = l-m, \widehat{Z}_{n-1} \leq m)}  +  \sum_{k=0}^{m-1} {\widehat{f}_n(k) \widehat{g}_{n-1} (l-k,m)}     \\
 = &  {\widehat{f}_n (m) \sum_{k=0}^m {\Pr (\widehat{Y}_{n-1} = l-m, \widehat{Z}_{n-1} = k)}}  +  \sum_{k=0}^{m-1} {\widehat{f}_n(k) \widehat{g}_{n-1} (l-k,m)}   \\
 = & \widehat{f}_n(m) \sum_{k=0}^m { \widehat{g}_{n-1} (l-m,k) }  +  \sum_{k=0}^{m-1} {\widehat{f}_n(k) \widehat{g}_{n-1} (l-k,m)} ,  
\end{split}
\]
for all $n \in \N \setminus \{1\}$ and for all $(l,m) \in \N_0^2$. Equivalently, we can restrict its domain to $\widehat{\mathcal{D}}_n$ since the joint PMF $\widehat{g}_n (l,m)$ definitely vanishes outside this region. 

Finally, using the domain $\widehat{\mathcal{D}}_{n-1} = \{ (l,m) \in \N_0^2 : m \leq l \leq (n-1)m \}$ of $\widehat{g}_{n-1} (l,m)$, we have 
\[
\begin{split}
\widehat{g}_n (l,m) = & \widehat{f}_n(m) \sum_{k=0}^m { \widehat{g}_{n-1} (l-m,k) }  +  \sum_{k=0}^{m-1} {\widehat{f}_n(k) \widehat{g}_{n-1} (l-k,m)}  \\
= & \widehat{f}_n(m) \sum_{k=\max\left({\left\lceil \tfrac{l-m}{n-1} \right\rceil},0\right)}^{\min(l-m,m)} { \widehat{g}_{n-1} (l-m,k) }  +  \sum_{k=\max\left( l-(n-1)m,0 \right)}^{\min(l-m,m-1)} {\widehat{f}_n(k) \widehat{g}_{n-1} (l-k,m)}       \\
= & \widehat{f}_n(m) \sum_{k=\left\lceil \tfrac{l-m}{n-1} \right\rceil}^{\min(l-m,m)} { \widehat{g}_{n-1} (l-m,k) }  +  \sum_{k=\max\left( l-(n-1)m,0 \right)}^{\min(l-m,m-1)} {\widehat{f}_n(k) \widehat{g}_{n-1} (l-k,m)} ,
\end{split}
\]
for all $n \in \N \setminus \{1\}$ and for all $(l,m) \in \widehat{\mathcal{D}}_n$, where the last equality follows from the fact that $l \geq m$ on $\widehat{\mathcal{D}}_n$ and the ceiling function is monotonically nondecreasing.
\end{proof}

\begin{figure}[!t]
\centering
\includegraphics[width=0.9\linewidth]{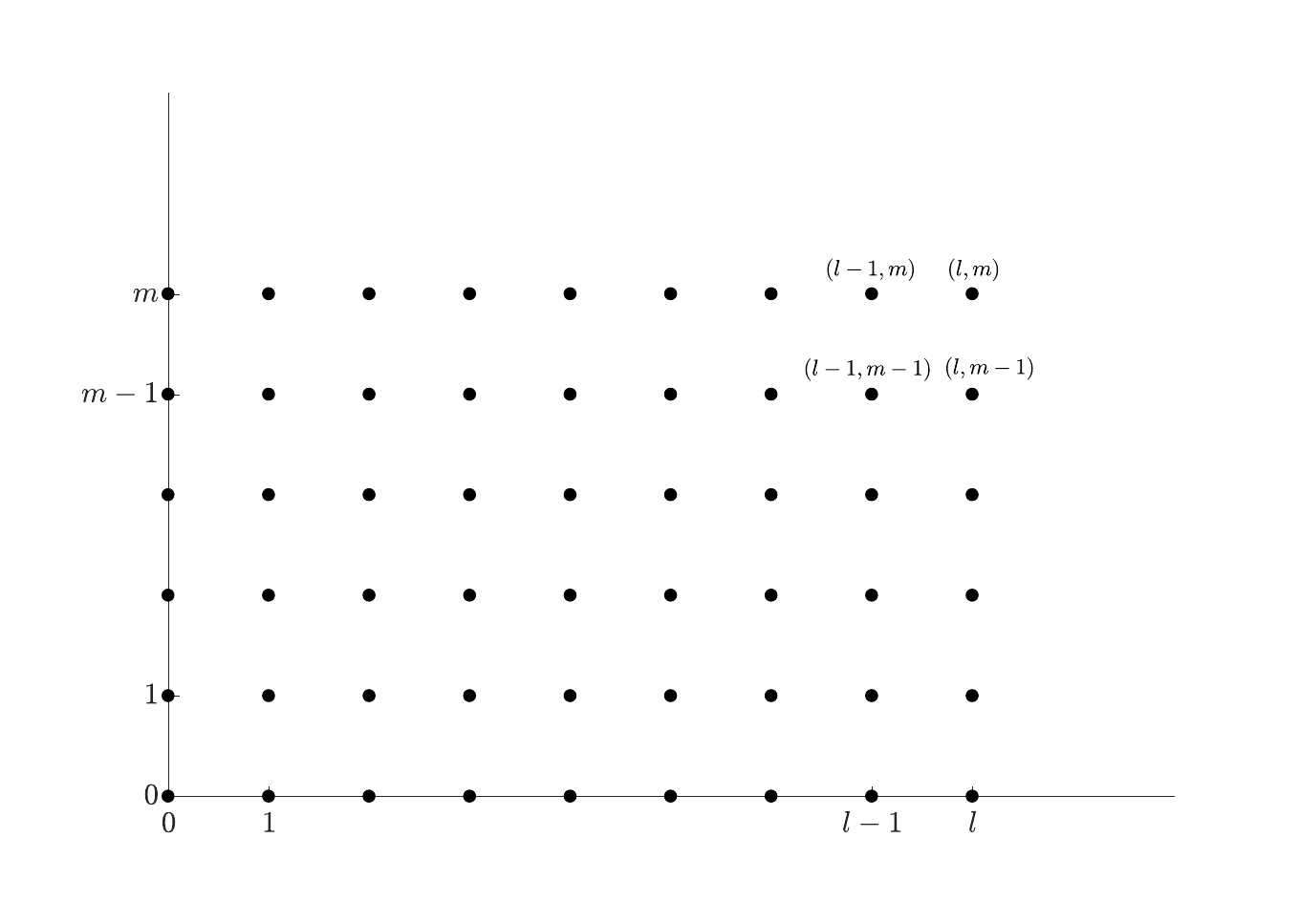} 
\caption{Two-dimensional lattice: each point $(l,m) \in \Z^2$ and each set of points $\mathcal{S} (l,m) \defeq \{(l',m') \in \Z^2 : 0 \leq l' \leq l,\, 0 \leq m' \leq m \}$  correspond to the probabilities $\widehat{g}_\tau (l,m)$ and $\widehat{G}_\tau (l,m) = \sum_{l'=0}^{l} \sum_{m'=0}^{m} {\widehat{g}_\tau (l',m')}$, respectively.}
\label{figure:2D_lattice}
\end{figure}

\break
\begin{remark}
An alternative proof of \eqref{equation:recurrence relation of joint PMF} using the recurrence relation of the joint CDF \eqref{equation:recurrence relation of joint CDF_discrete} is the following. Based on the two-dimensional lattice shown in Figure~\ref{figure:2D_lattice}, $\widehat{g}_\tau (l,m) \defeq \Pr (\widehat{Y}_\tau = l,\widehat{Z}_\tau = m)$ can be written as 
\begin{equation} \label{equation:g_tau} 
\begin{split}
\widehat{g}_\tau (l,m) & = \widehat{G}_\tau (l,m) - \widehat{G}_\tau (l,m-1) - \sum_{l'=0}^{l-1} {\widehat{g}_\tau (l',m)}    \\ 
& = \widehat{G}_\tau (l,m) - \widehat{G}_\tau (l,m-1) - \left( \widehat{G}_\tau (l-1,m) - \widehat{G}_\tau (l-1,m-1) \right)  \\
& = \widehat{G}_\tau (l,m) -  \widehat{G}_\tau (l-1,m) + \widehat{G}_\tau (l-1,m-1) - \widehat{G}_\tau (l,m-1)   , \\
& \:\, \quad  \forall \tau \in \N , \; \forall (l,m) \in \Z^2  . 
\end{split} 
\end{equation}  
By exploiting \eqref{equation:g_tau} for $\tau = n$ and leveraging \eqref{equation:recurrence relation of joint CDF_discrete}, which also holds for all $(y,z) \in \R^2$, we obtain 
\[ 
\begin{split}
\widehat{g}_n (l,m) & = \sum_{k=0}^{m} {\widehat{f}_n(k) \left( \widehat{G}_{n-1}(l - k,m) - \widehat{G}_{n-1}(l - 1 - k,m) \right)} \\
& \quad\,  +  \sum_{k=0}^{m-1} {\widehat{f}_n(k) \left( \widehat{G}_{n-1}(l - 1 - k,m-1) - \widehat{G}_{n-1}(l - k,m-1) \right)}     \\ 
& = {\widehat{f}_n(m) \left( \widehat{G}_{n-1}(l - m,m) - \widehat{G}_{n-1}(l - m - 1,m) \right)}    \\
& \quad\,  +  \sum_{k=0}^{m-1} {\widehat{f}_n(k) \left( \widehat{G}_{n-1}(l - k,m) - \widehat{G}_{n-1}(l - k - 1,m) \right.} \\
& \quad\,  + {\left. \widehat{G}_{n-1}(l - k - 1,m-1) - \widehat{G}_{n-1}(l - k,m-1) \right)} \\
& = \widehat{f}_n(m) \sum_{k=0}^m { \widehat{g}_{n-1} (l-m,k) }  +  \sum_{k=0}^{m-1} {\widehat{f}_n(k) \widehat{g}_{n-1} (l-k,m)}  ,  \\
& \:\, \quad  \forall n \in \N \setminus \{1\} , \; \forall (l,m) \in \Z^2  ,  
\end{split}
\]
where the last equality follows from \eqref{equation:g_tau} for $\tau = n-1 \geq 1$ and the fact that 
\[
\widehat{G}_{n-1}(l,m) - \widehat{G}_{n-1}(l-1,m) = \sum_{k=0}^m {\widehat{g}_{n-1} (l,k)},   \quad  \forall (l,m) \in \Z^2 .
\]
This can be easily seen from the lattice representation in Figure~\ref{figure:2D_lattice} with $\tau = n-1$. As a consequence, we get \eqref{equation:recurrence relation of joint PMF} after shrinking the domain to $\widehat{\mathcal{D}}_n$.  
\end{remark}

\begin{remark} \label{remark:Fundamental difference}
It is emphasized that \eqref{equation:recurrence relation of joint PMF} has a \emph{fundamental difference} compared to its continuous counterpart in \eqref{equation:recurrence relation of joint PDF}: the upper limit of the second sum is equal to $m-1$ instead of $m$. As a result, we \emph{cannot} derive \eqref{equation:recurrence relation of joint PMF} from \eqref{equation:recurrence relation of joint PDF} by just making the substitutions $(x,y,z) \mapsto (k,l,m)$, $\dif x \mapsto 1$, and by replacing integrals with sums. 
\end{remark}

\subsubsection{\textnormal{\textbf{Independent and Identically Distributed Random Variables.}}}

The following lemma is analogous to Lemma~\ref{lemma:Integration property of joint PDF} and will be helpful in order to prove another recurrence relation for the i.i.d.~case (in addition to the straightforward application of \eqref{equation:recurrence relation of joint PMF} with $\widehat{f}_i(k) = \widehat{f}(k)$ for all $i \in \N$).  

\begin{lemma} \label{lemma:Summation property of joint PMF}
For all $n \in \N \setminus \{1\}$ and for all $(l,m) \in \widehat{\mathcal{D}}_{n+1}$, it holds that  
\[
\sum_{k=0}^m  {\widehat{g}_n(l-m,k)}  =  \sum_{k=0}^m {\widehat{f}_n(k) \sum_{\nu=0}^m {\widehat{g}_{n-1}(l-m-k,\nu)}} .
\]
\end{lemma}

\begin{proof}
By leveraging \eqref{equation:recurrence relation of joint PMF}, we get 
\[
\widehat{g}_n(l-m,k) =  \widehat{f}_n(k) \sum_{\nu=0}^k {\widehat{g}_{n-1}(l-m-k,\nu)} + \sum_{\nu=0}^{k-1} {\widehat{f}_n(\nu) \widehat{g}_{n-1}(l-m-\nu,k)}    ,
\]
on the set $\{(k,l,m) \in \Z^3 : k \geq 0, \; l \geq m, \; k \leq l-m \leq nk\} = \{(k,l,m) \in \Z^3 : \break k \geq 0, \; m+k \leq l \leq m+nk\}$. Then, by summing for $0 \leq k \leq m$, we have 
\[
\begin{split}
\sum_{k=0}^m  {\widehat{g}_n(l-m,k)}  &=   \sum_{k=0}^m {\widehat{f}_n(k) \sum_{\nu=0}^k {\widehat{g}_{n-1}(l-m-k,\nu)}}  +  \sum_{k=0}^m {\sum_{\nu=0}^{k-1} {\widehat{f}_n(\nu) \widehat{g}_{n-1}(l-m-\nu,k)}}  \\
& \mathop{=}^{(e)} \sum_{k=0}^m {\widehat{f}_n(k) \sum_{\nu=0}^k {\widehat{g}_{n-1}(l-m-k,\nu)}}  +  \sum_{\nu=0}^{m-1} {\widehat{f}_n(\nu) \sum_{k=\nu+1}^m {\widehat{g}_{n-1}(l-m-\nu,k)}}   \\ 
& \mathop{=}^{(f)}  \sum_{k=0}^m {\widehat{f}_n(k) \sum_{\nu=0}^k {\widehat{g}_{n-1}(l-m-k,\nu)}}  +  \sum_{k=0}^{m-1} {\widehat{f}_n(k) \sum_{\nu=k+1}^m {\widehat{g}_{n-1}(l-m-k,\nu)}}     \\ 
& = \widehat{f}_n(m) \sum_{\nu=0}^m {\widehat{g}_{n-1}(l-2m,\nu)} + \sum_{k=0}^{m-1} {\widehat{f}_n(k) \left( \sum_{\nu=0}^k {\widehat{g}_{n-1}(l-m-k,\nu)} \right.}   \\
& \quad\:  {\left. + \sum_{\nu=k+1}^m {\widehat{g}_{n-1}(l-m-k,\nu)} \right)}   \\ 
& =  \widehat{f}_n(m) \sum_{\nu=0}^m {\widehat{g}_{n-1}(l-2m,\nu)} + \sum_{k=0}^{m-1} {\widehat{f}_n(k) \sum_{\nu=0}^m {\widehat{g}_{n-1}(l-m-k,\nu)}}   \\
& =  \sum_{k=0}^m {\widehat{f}_n(k) \sum_{\nu=0}^m {\widehat{g}_{n-1}(l-m-k,\nu)}}   ,   
\end{split}   
\]
for all $(l,m) \in \N_0^2$ such that $m \leq l \leq (n+1)m$, that is, for all $(l,m) \in \widehat{\mathcal{D}}_{n+1}$. Note that equalities $(e)$ and $(f)$ have been derived by interchanging the order of summation ($\{0 \leq k \leq m\} \land \{0 \leq \nu \leq k-1\} \iff \{0 \leq \nu \leq m-1\} \land \{\nu + 1 \leq k \leq m\}$) and by interchanging the variables $k$, $\nu$ in the second double sum, respectively.  
\end{proof}

\begin{theorem} \label{theorem:recurrence relation of joint PMF_iid}
Let $\widehat{f}(k)$ be the common PMF of random variables $\{\widehat{X}_i\}_{i \in \N}$, that is, \linebreak $\widehat{f}_i(k) = \widehat{f}(k)$, $\forall i \in \N$. Furthermore, let  $\widehat{h}_n(l,m)$ be defined by the following recursive formula 
\begin{equation} \label{equation:recursive definition of h}
\begin{split}
\widehat{h}_n(l,m) \defeq &  \sum_{k=0}^m {\widehat{f}(k) \widehat{h}_{n-1}(l-k,m)} + {\widehat{f}(m) \widehat{g}_{n-1} (l-m,m)}  \\
 = & \sum_{k=\max\left( l-(n-1)m,0 \right)}^{\min(l-m,m)} {\widehat{f}(k) \widehat{h}_{n-1}(l-k,m)} + {\widehat{f}(m) \widehat{g}_{n-1} (l-m,m)}  ,  \\
& \: \forall n \in \N \setminus \{1\}, \; \forall (l,m) \in \widehat{\mathcal{D}}_n ,
\end{split}
\end{equation}
with initial condition $\widehat{h}_1(l,m) = 0$, $\forall (l,m) \in \widehat{\mathcal{D}}_1$. Then, the PMF of $(\widehat{Y}_n,\widehat{Z}_n)$ is given by the recurrence relation 
\begin{equation} \label{equation:recurrence relation of joint PMF_iid}
\begin{split}
\widehat{g}_n (l,m) & = n \widehat{f}(m) \sum_{k=0}^{m} { \widehat{g}_{n-1} (l-m,k) } - \widehat{h}_n(l,m) \\
& = n \widehat{f}(m) \sum_{k=\left\lceil \tfrac{l-m}{n-1} \right\rceil}^{\min(l-m,m)} { \widehat{g}_{n-1} (l-m,k) } - \widehat{h}_n(l,m)   ,  \\
&  \quad\:\,  \forall n \in \N \setminus \{1\}, \; \forall (l,m) \in \widehat{\mathcal{D}}_n ,  
\end{split} 
\end{equation}
with initial condition $\widehat{g}_1 (l,m) = \widehat{f}(l)  = \widehat{f}(m)$, $\forall (l,m) \in \widehat{\mathcal{D}}_1$. 
\end{theorem}

\begin{proof}
The initial condition $\widehat{g}_1(l,m)$ can be easily derived from \eqref{equation:initial condition for joint PMF}. Now, we will prove \eqref{equation:recurrence relation of joint PMF_iid} using induction on $n$ and given the recursive definition of $\widehat{h}_n(l,m)$ in \eqref{equation:recursive definition of h}.  

Basis step: For $n=2$, \eqref{equation:recurrence relation of joint PMF} implies that 
\[
\begin{split}
\widehat{g}_2 (l,m) & = \widehat{f}(m) \sum_{k=0}^m { \widehat{g}_1 (l-m,k) }  +  \sum_{k=0}^m {\widehat{f}(k) \widehat{g}_1 (l-k,m)} - \widehat{f}(m) \widehat{g}_1 (l-m,m)  \\
& = \widehat{f}(m) \sum_{k=0}^m { \widehat{g}_1 (l-m,k) } + \widehat{f}(m) \sum_{k=0}^m {\widehat{f}(k) [l-k=m] } - \widehat{h}_2(l,m)   \\ 
& =  \widehat{f}(m) \sum_{k=0}^m { \widehat{g}_1 (l-m,k) } + \widehat{f}(m) \sum_{k=0}^{m} {\widehat{f}(k) [l-m=k]} - \widehat{h}_2(l,m)    \\
& =  2 \widehat{f}(m) \sum_{k=0}^{m} { \widehat{g}_1 (l-m,k) } - \widehat{h}_2(l,m)  , \quad \forall (l,m) \in \widehat{\mathcal{D}}_2 ,
\end{split}
\] 
where we have used the fact that $\widehat{h}_2(l,m) = \sum_{k=0}^m {\widehat{f}(k) \widehat{h}_1(l-k,m)} + \widehat{f}(m) \widehat{g}_1 (l-m,m) = \widehat{f}(m) \widehat{g}_1 (l-m,m)$, $\forall (l,m) \in \widehat{\mathcal{D}}_2$, according to \eqref{equation:recursive definition of h} and its initial condition. As a result, \eqref{equation:recurrence relation of joint PMF_iid} is true for $n=2$. 

Inductive step: Assume that 
\[
\widehat{g}_{\tau-1} (l,m) = (\tau-1) \widehat{f}(m) \sum_{\nu=0}^{m} { \widehat{g}_{\tau-2} (l-m,\nu) } - \widehat{h}_{\tau-1}(l,m) , \quad \forall (l,m) \in \widehat{\mathcal{D}}_{\tau-1}  ,
\]
for some arbitrary integer $\tau \geq 3$ (inductive hypothesis). Then, \eqref{equation:recurrence relation of joint PMF} for $n=\tau$, Lemma~\ref{lemma:Summation property of joint PMF} for $n=\tau-1$, and \eqref{equation:recursive definition of h} for $n=\tau$ give  
\[
\begin{split}
\widehat{g}_\tau (l,m) &= \widehat{f}(m) \sum_{k=0}^m { \widehat{g}_{\tau-1} (l-m,k) }  +  \sum_{k=0}^{m} {\widehat{f}(k) \widehat{g}_{\tau-1} (l-k,m)} - {\widehat{f}(m) \widehat{g}_{\tau-1} (l-m,m)}  \\
& = \widehat{f}(m) \sum_{k=0}^m { \widehat{g}_{\tau-1} (l-m,k) }  + (\tau-1) \widehat{f}(m) \sum_{k=0}^m {\widehat{f}(k) \sum_{\nu=0}^m {\widehat{g}_{\tau-2}(l-k-m,\nu)}}    \\
& \quad\,  - \sum_{k=0}^m {\widehat{f}(k) \widehat{h}_{\tau-1}(l-k,m)} - \widehat{f}(m) \widehat{g}_{\tau-1} (l-m,m)  \\
& =  \widehat{f}(m) \sum_{k=0}^{m} { \widehat{g}_{\tau-1} (l-m,k) } + (\tau-1) \widehat{f}(m) \sum_{k=0}^{m} { \widehat{g}_{\tau-1} (l-m,k) } - \widehat{h}_\tau (l,m)   \\
& =  \tau \widehat{f}(m) \sum_{k=0}^{m} { \widehat{g}_{\tau-1} (l-m,k) } - \widehat{h}_\tau(l,m) , \quad \forall (l,m) \in \widehat{\mathcal{D}}_\tau  .
\end{split}  
\]
Note that the second equality is true on $\widehat{\mathcal{D}}_\tau$, since $\widehat{g}_{\tau-1} (l-k,m) = (\tau-1) \widehat{f}(m) \break \sum_{\nu=0}^{m} { \widehat{g}_{\tau-2} (l-k-m,\nu) } - \widehat{h}_{\tau-1}(l-k,m)$ on $\{(k,l,m) \in \Z^3 : m \geq 0, \; l \geq k, \; m \leq l-k \leq (\tau-1)m\} = \{(k,l,m) \in \Z^3 : m \geq 0, \; k+m \leq l \leq k+(\tau-1)m\}$ and the domain of summation is $0 \leq k \leq m$. Therefore, \eqref{equation:recurrence relation of joint PMF_iid} holds for all integers $n \geq 2$.
\end{proof}

\begin{remark}
The direct application of \eqref{equation:recurrence relation of joint PMF} with $\widehat{f}_i(k) = \widehat{f}(k)$, $\forall i \in \N$, certainly gives a simpler recurrence relation than \eqref{equation:recurrence relation of joint PMF_iid}. Nevertheless, Theorem~\ref{theorem:recurrence relation of joint PMF_iid} in comparison with Theorem~\ref{theorem:recurrence relation of joint PDF_iid} shows that there is a \emph{key difference} between discrete and continuous random variables: the term $\widehat{h}_n(l,m)$ in \eqref{equation:recurrence relation of joint PMF_iid} is \emph{not} present at all in \eqref{equation:recurrence relation of joint PDF_iid}; see also Remark~\ref{remark:Fundamental difference}. 
\end{remark}

\subsection{Application:~Peak-to-Average Ratio}

The peak-to-average ratio of the $n^\text{th}$ order is analogously defined by 
\[
\widehat{\rho}_n \defeq \frac{\max_{1 \leq i \leq n} \widehat{X}_i}{\tfrac{1}{n} \sum_{i=1}^n {\widehat{X}_i}} = \frac{\widehat{Z}_n}{\tfrac{1}{n} \widehat{Y}_n} , \quad \text{whenever}\ \widehat{Y}_n \neq 0  .
\]  
The probability that $\widehat{\rho}_n$ belongs to the interval $[\alpha,\beta]$, $0 \leq \alpha \leq \beta$, is equal to  
\begin{equation}
\begin{split} 
\Pr (\alpha \leq \widehat{\rho}_n \leq \beta) & = \Pr \left(\widehat{Y}_n \geq 1, \left\lceil \tfrac{\alpha}{n} \widehat{Y}_n \right\rceil \leq \widehat{Z}_n \leq \left\lfloor \tfrac{\beta}{n} \widehat{Y}_n \right\rfloor \right) = \sum_{l=1}^{\infty} { \sum_{m=\left\lceil \tfrac{\alpha}{n} l \right\rceil}^{\left\lfloor \tfrac{\beta}{n} l \right\rfloor} {\widehat{g}_n (l,m)} }   \\
& =  \sum_{l=1}^{\infty} { \sum_{m=\max \left( \left\lceil \tfrac{\alpha}{n} l \right\rceil , \left\lceil \tfrac{1}{n} l \right\rceil \right)}^{\min \left( \left\lfloor \tfrac{\beta}{n} l \right\rfloor , l \right)} {\widehat{g}_n (l,m)} }  =  \sum_{l=1}^{\infty} { \sum_{m=\left\lceil \tfrac{\max(\alpha,1)}{n} l \right\rceil}^{\left\lfloor \min \left( \tfrac{\beta}{n} , 1 \right) l \right\rfloor} {\widehat{g}_n (l,m)} }   ,
\end{split} 
\end{equation}
where $\widehat{g}_n(l,m)$ is computed recursively based on Theorem~\ref{theorem:Joint PMF}.

\subsection{Other Applications}

By leveraging the recurrence relation of the joint PMF in Theorem~\ref{theorem:Joint PMF}, we can calculate various probabilistic quantities. Indicative examples include the marginal PMFs 
\[f_{\widehat{Y}_n}(l) = \sum_{m=0}^\infty {\widehat{g}_n(l,m)} \quad \text{and} \quad f_{\widehat{Z}_n}(m) = \sum_{l=0}^\infty {\widehat{g}_n(l,m)},\] 
the conditional PMFs 
\[f_{\widehat{Y}_n|\widehat{Z}_n} (l|m) = \frac{\widehat{g}_n(l,m)}{f_{\widehat{Z}_n}(m)} = \frac{\widehat{g}_n(l,m)}{\sum_{l'=0}^\infty {\widehat{g}_n(l',m)}} \quad \text{and} \quad f_{\widehat{Z}_n|\widehat{Y}_n} (m|l) = \frac{\widehat{g}_n(l,m)}{f_{\widehat{Y}_n}(l)} = \frac{\widehat{g}_n(l,m)}{\sum_{m'=0}^\infty {\widehat{g}_n(l,m')}},\] 
as well as the expectation 
\[\mathbb{E} (h(\widehat{Y}_n,\widehat{Z}_n)) = \sum_{m=0}^\infty {\sum_{l=0}^\infty  h(l,m) \widehat{g}_n(l,m)} = \sum_{m=0}^\infty {\sum_{l=m}^{nm}  h(l,m) \widehat{g}_n(l,m)},\] 
where $h(l,m)$ is a real-valued function defined on $\N_0^2$ such that the last double sum exists. Note that the joint moment $\mathbb{E} (\widehat{Y}_n^\alpha \widehat{Z}_n^\beta)$, with $\alpha, \beta \in \R$, can be obtained from the latter formula by setting $h(l,m) = l^\alpha m^\beta$, provided that the double sum is defined.

\section{Extensions and Generalizations}
\label{section:Extensions and Generalizations}

\subsection{Continuous/Discrete Random Variables with Negative Values}

Let \linebreak $\{X'_i\}_{i \in \N}$ be a collection of independent, not necessarily identically distributed, continuous random variables such that $X'_i \geq -c$, for all $i \in \N$, where $c>0$ is a finite constant. Suppose that their PDFs $\{f'_i(x)\}_{i \in \N}$ are \emph{continuous} for $x \geq -c$, and equal to zero for $x < -c$; here, do \emph{not} confuse $f'_i(x)$ with $\od{}{x} f_i(x)$. By applying the variable transformation $X_i \defeq X'_i + c$, for all $i \in \N$, we can observe that each $X_i \geq 0$ and its PDF $f_i(x) = f'_{i}(x-c)$ is continuous on $\R_+$. As a result, the joint CDF of the sum $Y'_n \defeq \sum_{i=1}^n {X'_i}$ and the maximum $Z'_n \defeq \max_{1 \leq i \leq n} {X'_i}$ is written as  
\begin{equation}
G'_n (y,z) \defeq \Pr (Y'_n \leq y,Z'_n \leq z) = \Pr (Y_n \leq y + nc,Z_n \leq z + c) = G_n(y + nc, z + c)  ,
\end{equation}
where $Y_n \defeq \sum_{i=1}^n {X_i} = Y'_n + nc$, $Z_n \defeq \max_{1 \leq i \leq n} {X_i} = Z'_n + c$, and $G_n(\cdot,\cdot)$ is given by Theorem~\ref{theorem:Recurrence relation and continuity of joint CDF}. Moreover, the joint PDF of $Y'_n$ and $Z'_n$ is computed as follows 
\begin{equation}
g'_n (y,z) = {\dmd{}{2}{y}{}{z}{} G'_n(y,z)} =  g_n(y + nc, z + c)  ,
\end{equation} 
where $g_n(\cdot,\cdot)$ is given by Theorem~\ref{theorem:Joint PDF}. 

Similarly in the discrete case, let $\{\widehat{X}'_i\}_{i \in \N}$ be a collection of independent, not necessarily identically distributed, \emph{integer-valued} random variables such that $\widehat{X}'_i \geq -\lambda$, for all $i \in \N$ and for some fixed $\lambda \in \N$, with PMFs $\{\widehat{f}'_i (k)\}_{i \in \N}$. The variable transformation $\widehat{X}_i \defeq \widehat{X}'_i + \lambda$, for all $i \in \N$, implies that every $\widehat{X}_i \in \N_0$ and its PMF is $\widehat{f}_i(k) = \widehat{f}'_i (k-\lambda)$. Therefore, the joint CDF of the sum $\widehat{Y}'_n \defeq \sum_{i=1}^n {\widehat{X}'_i}$ and the maximum $\widehat{Z}'_n \defeq \max_{1 \leq i \leq n} {\widehat{X}'_i}$ is expressed as  
\begin{equation}
\widehat{G}'_n (y,z) \defeq \Pr (\widehat{Y}'_n \leq y,\widehat{Z}'_n \leq z) = \Pr (\widehat{Y}_n \leq y + n\lambda,\widehat{Z}_n \leq z + \lambda) = \widehat{G}_n(y + n\lambda, z + \lambda)  ,
\end{equation}
where $\widehat{Y}_n \defeq \sum_{i=1}^n {\widehat{X}_i} = \widehat{Y}'_n + n\lambda$, $\widehat{Z}_n \defeq \max_{1 \leq i \leq n} {\widehat{X}_i} = \widehat{Z}'_n + \lambda$, and $\widehat{G}_n(\cdot,\cdot)$ is given by Theorem~\ref{theorem:Recurrence relation of joint CDF_discrete}. In addition, the joint PMF of $\widehat{Y}'_n$ and $\widehat{Z}'_n$ is calculated as follows 
\begin{equation}
\widehat{g}'_n (l,m) \defeq \Pr (\widehat{Y}'_n = l,\widehat{Z}'_n = m) = \Pr (\widehat{Y}_n = l + n\lambda,\widehat{Z}_n = m + \lambda) = \widehat{g}_n (l + n\lambda, m + \lambda) ,
\end{equation}
where $\widehat{g}_n(\cdot,\cdot)$ is given by Theorem~\ref{theorem:Joint PMF}.

\subsection{Continuous Random Variables with Discontinuous PDFs}

We consider a collection $\{X_i\}_{i \in \N}$ of independent, not necessarily identically distributed, nonnegative continuous random variables. Now, suppose that each PDF $f_i(x)$ is \emph{bounded and piecewise continuous} on $\R_+$, i.e., $(0 \leq)\, f_i(x) \leq c_{f_i} < \infty$ for some real constant $c_{f_i} > 0$, and $f_i(x)$ has a finite number of discontinuities in $\R_+$. Therefore, $f_i(x)$ is continuous almost everywhere on $\R_+$.  

Firstly, Theorem~\ref{theorem:Recurrence relation and continuity of joint CDF} is still valid. Now, the only difference in its proof is about the continuity of $G_n(y,z)$ in the inductive step: if $G_{k-1}(y,z)$ is continuous on $\R_+^2$ then so is $G_k (y,z)$, because the product of a continuous and an almost-everywhere continuous function is continuous almost everywhere, and the integral of any almost-everywhere continuous function is continuous.   

Secondly, \eqref{equation:recurrence relation of joint PDF}~and~\eqref{equation:initial condition for joint PDF} in Theorem~\ref{theorem:Joint PDF} remain true, however the last sentence about the continuity of $g_n(y,z)$ is \emph{not} valid in general and must be replaced by  \vspace{3mm}
\newline \emph{``Furthermore, $g_n (y,z)$ is continuous almost everywhere and bounded on its domain $\mathcal{D}_n$, and the integral $\int_0^z {g_n (y - z,x) \dif x}$ is bounded on $\mathcal{D}_{n+1}$, for all $n \in \N \setminus \{1\}$.''}  \vspace{3mm}
\newline Based on the original proof of Theorem~\ref{theorem:Joint PDF}, we give a sketch of the proof for the modified version, where only the differences are discussed. 

Basis step: For $n=2$ and using the fundamental theorem of calculus for Lebesgue integrals~\cite[pp.~105-106, Corollary~3.33 and Theorem~3.35]{Folland}, instead of the classical one, we obtain the equation immediately before \eqref{equation:g_2}, which now holds almost everywhere; note that the function $F_1(y-x) f_2(x)$ is almost-everywhere continuous, and so integrable. Hence, we can select the PDF $g_2(y,z)$ to be given by \eqref{equation:g_2}, because any two such PDFs are equal almost everywhere. From the latter equation, $g_2 (y,z)$ is continuous almost everywhere on $\mathcal{D}_2$, since the sum/product of two almost-everywhere continuous functions is continuous almost everywhere. Also, we have $g_2 (y,z) =  f_1(y-z) f_2(z) + f_1(z) f_2(y-z) \leq 2 c_{f_1} c_{f_2} < \infty$, so $g_2 (y,z)$ is bounded on $\mathcal{D}_2$. By following the proof of Lemma~\ref{lemma:Integration property of joint PDF}, we know that $\int_0^z  {g_2(y-z,x) \dif x}  =  \int_0^z {f_2(x) \int_0^z {g_1(y-z-x,w) \dif w} \dif x} \leq c_{f_2} \int_0^z {\int_{y-2z}^{y-z} {g_1(x',w) \dif x'} \dif w} \leq c_{f_2} \int_0^\infty {\int_0^\infty {g_1(x,w) \dif x} \dif w} = c_{f_2} <\infty$, so $\int_0^z  {g_2(y-z,x) \dif x}$ is bounded on $\mathcal{D}_3$. 

Inductive step: Suppose that $g_{k-1}(y,z)$ is continuous almost everywhere and bounded (by a real constant $c_{g_{k-1}} > 0$) on $\mathcal{D}_{k-1}$, and the integral $\int_0^z {g_{k-1} (y - z,x) \dif x}$ is bounded (by a real constant $c_{\int \!\! g_{k-1}} > 0$) on $\mathcal{D}_k$, for some integer $k \geq 3$. Now, we should use Proposition~\ref{proposition:Partial Differentiation Under the Integral Sign_modified version} instead of Proposition~\ref{proposition:Partial Differentiation Under the Integral Sign}, without any continuity assumption. In particular, observe that all the assumptions of Proposition~\ref{proposition:Partial Differentiation Under the Integral Sign_modified version} are satisfied: 
\begin{enumerate}
\item[1)] $0 \leq h(x,y,z)=f_k(x) G_{k-1}(y - x,z) \leq f_k(x)$, since $G_{k-1}(y,z) \leq 1$, and thus $\int_0^z {h(x,y,z) \dif x} \leq \int_0^z {f_k(x) \dif x} \leq 1 < \infty$,    \vspace{2mm}
\item[2)] $0 \leq \pd{}{z} h(x,y,z) = f_k(x) \int_0^{y-x} {g_{k-1} (y',z) \dif y'} \leq c_{g_{k-1}} f_k(x) (y-x) \leq c_{g_{k-1}} f_k(x) y \eqdef \varphi(x,y)$, for $0 \leq x \leq y$, with $\int_0^z \varphi(x,y) \dif x = c_{g_{k-1}} y \int_0^z f_k(x) \dif x \leq c_{g_{k-1}} y < \infty$, and    \vspace{2mm}
\item[3)] $0 \leq {\md{}{2}{y}{}{z}{} h(x,y,z)} = f_k(x) g_{k-1} (y - x,z) \leq c_{g_{k-1}} f_k(x) \eqdef \phi(x)$, with $\int_0^z \phi(x) \dif x = c_{g_{k-1}} \int_0^z f_k(x) \dif x \leq c_{g_{k-1}} < \infty$. 
\end{enumerate}    
As a result, we can choose $g_k(y,z) = f_k(z) \int_0^z {g_{k-1} (y - z,x) \dif x} + \int_0^z {f_k(x) g_{k-1} (y - x,z) \dif x} \leq c_{f_k} c_{\int \!\! g_{k-1}} + c_{g_{k-1}} \int_0^z {f_k(x) \dif x} \leq c_{f_k} c_{\int \!\! g_{k-1}} + c_{g_{k-1}} < \infty$, so $g_k(y,z)$ is bounded on $\mathcal{D}_k$. From the proof of Lemma~\ref{lemma:Integration property of joint PDF}, we obtain  $\int_0^z  {g_k(y-z,x) \dif x}  =  \int_0^z {f_k(x) \int_0^z {g_{k-1}(y-z-x,w) \dif w} \dif x} \leq c_{f_k} \int_0^z {\int_{y-2z}^{y-z} {g_{k-1}(x',w) \dif x'} \dif w} \leq c_{f_k} \int_0^\infty {\int_0^\infty {g_{k-1}(x,w) \dif x} \dif w} = c_{f_k} < \infty$, so $\int_0^z  {g_k(y-z,x) \dif x}$ is bounded on $\mathcal{D}_{k+1}$. Moreover, $g_k(y,z)$ is continuous almost everywhere on $\mathcal{D}_k$, because it is the sum of two almost-everywhere continuous functions. Hence, the claim has been proved.

In summary, ``almost all'' the results of Section~\ref{section:Continuous Random Variables} remain valid, except for the aforementioned slight/minor modification of Theorem~\ref{theorem:Joint PDF}. The case of discontinuous PDFs is a generalization of the case of continuous PDFs (no discontinuities in $\R_+$ at all). This is consistent with the fact that, for the joint PDF $g_n(y,z)$, continuity (everywhere) implies continuity almost everywhere, since the empty set has measure zero by definition.

\subsection{Mixture of Continuous and Discrete Random Variables}

Let $\{\widetilde{X}_i\}_{i \in \N}$ be a mixture of independent, not necessarily identically distributed, nonnegative continuous and discrete random variables (i.e., each $\widetilde{X}_i$ is either continuous or discrete), with PDFs/PMFs $\{\widetilde{f}_i(x)\}_{i \in \N}$ and CDFs $\{\widetilde{F}_i(x)\}_{i \in \N}$. Furthermore, assume that every continuous random variable has a \emph{continuous} (or, \emph{bounded and piecewise continuous}) PDF on $\R_+$, and that each discrete random variable takes \emph{integer values}. Then, the joint CDF $\widetilde{G}_n(y,z) \defeq \Pr (\widetilde{Y}_n \leq y,\widetilde{Z}_n \leq z)$ of the sum $\widetilde{Y}_n \defeq \sum_{i=1}^n {\widetilde{X}_i}$ and the maximum $\widetilde{Z}_n \defeq \max_{1 \leq i \leq n} {\widetilde{X}_i}$ is given by
\begin{equation}
\begin{split}
\widetilde{G}_n(y,z) & = \left\{  
\begin{array}{ll}
      \int_0^{\min(y,z)} {\widetilde{G}_{n-1}(y - x,z) \widetilde{f}_n(x) \dif x}, &  \text{if $\widetilde{X}_n$ is continuous} , \\
      \sum_{k=0}^{\min(\lfloor y \rfloor,\lfloor z \rfloor)} {\widetilde{G}_{n-1}(y - k,z) \widetilde{f}_n(k)}, & \text{if $\widetilde{X}_n$ is discrete} , \\ 
\end{array} 
\right. \\
& \:\; \quad \forall n \in \N \setminus \{1\} , \; \forall (y,z) \in \R_+^2 , 
\end{split}
\end{equation}   
with initial condition 
\begin{equation}
\widetilde{G}_1(y,z) = \widetilde{F}_1 (\min(y,z)) ,  \quad \forall (y,z) \in \R_+^2 .
\end{equation}
In other words, this is a generalization of the pure continuous and discrete case presented in Sections~\ref{section:Continuous Random Variables}~and~\ref{section:Discrete Random Variables}, respectively. Finally, it is worth noting that $\widetilde{G}_n(y,z)$ is \emph{not} continuous on $\R_+^2$, when there exists an index $j \in \{1,\dots,n\}$ such that $\widetilde{X}_j$ is a discrete random variable.


\appendix

\section{Partial Differentiation Under the Integral Sign}
\label{appendix:Partial Differentiation Under the Integral Sign}

In this part of the paper, we provide a generalization of the Leibniz integral rule for partial differentiation of integrals with variable limits. For this purpose, the following two lemmas will be instrumental. 

\begin{lemma}[Leibniz's integral rule with constant limits of integration~{\cite[p.~276, Theorem~7.1]{Lang}}]
\label{lemma:Leibniz's integral rule_constant limits}
Let $f(x,t)$ and $\pd{}{t} f(x,t)$ be defined and continuous on $\mathcal{X} \times \mathcal{T}$, where $\mathcal{X} \defeq [x_1,x_2]$ and $\mathcal{T} \defeq [t_1,t_2]$ are subsets of $\R$ with $-\infty < x_1 \leq x_2 < \infty$ and $-\infty < t_1 < t_2 < \infty$. Moreover, let $a,b \in \mathcal{X}$ be two constants. Then, for all $t \in \mathcal{T}$, 
\begin{equation} \label{equation:Leibniz's integral rule_constant limits}
\dod{}{t} \left(  \int_{a}^{b} f(x,t) \dif x \right)  = {\int_{a}^{b} \dpd{}{t} f(x,t) \dif x}  .
\end{equation} 
\end{lemma}

\begin{lemma}[Leibniz's integral rule with variable limits of integration] 
\label{lemma:Leibniz's integral rule_variable limits}
Let $f(x,t)$ and $\pd{}{t} f(x,t)$ be defined and continuous on $\mathcal{X} \times \mathcal{T}$, where $\mathcal{X} \defeq [x_1,x_2]$ and $\mathcal{T} \defeq [t_1,t_2]$ are subsets of $\R$ with $-\infty < x_1 \leq x_2 < \infty$ and $-\infty < t_1 < t_2 < \infty$. In addition, let $a, b : \mathcal{T} \to \mathcal{X}$ be differentiable functions on $\mathcal{T}$. Then, for all $t \in \mathcal{T}$,
\begin{equation} \label{equation:Leibniz's integral rule_variable limits}
\dod{}{t} \left( \int_{a(t)}^{b(t)} f(x,t) \dif x \right) = {\eval[2]{f(x,t)}_{x=b(t)} \cdot \dod{}{t}} b(t) - {\eval[2]{f(x,t)}_{x=a(t)} \cdot \dod{}{t}} a(t) + {\int_{a(t)}^{b(t)} \dpd{}{t} f(x,t) \dif x} .
\end{equation}
\end{lemma}

\begin{proof} 
This is a consequence of the fundamental theorem of calculus, the basic form of Leibniz's integral rule (Lemma~\ref{lemma:Leibniz's integral rule_constant limits}), and the chain rule for multivariate functions. Let us define the function $I(u,v,t) \defeq \int_v^u {f(x,t) \dif x}$, for all $(u,v,t) \in \mathcal{X}^2 \times \mathcal{T}$. From the fundamental theorem of calculus, we have $\pd{}{u} I(u,v,t) = f(u,t)$ and $\pd{}{v} I(u,v,t) = - f(v,t)$, while $\pd{}{t} I(u,v,t) = \int_v^u {\pd{}{t} f(x,t) \dif x}$ due to Lemma~\ref{lemma:Leibniz's integral rule_constant limits}. Now, by applying the chain rule, we obtain 
\[
\begin{split}
\dod{}{t} \left( \int_{a(t)}^{b(t)} f(x,t) \dif x \right) & = \dod{}{t} I(b(t),a(t),t) = \eval[3]{\dpd{}{u} I(u,v,t)}_{u=b(t),v=a(t)} \cdot \dod{}{t} b(t)    \\
& \quad  +  {\eval[3]{\dpd{}{v} I(u,v,t)}_{u=b(t),v=a(t)} \cdot \dod{}{t} a(t)} + {\eval[3]{\dpd{}{t} I(u,v,t)}_{u=b(t),v=a(t)} \cdot \dod{}{t} t}    \\
& = {\eval[2]{f(u,t)}_{u=b(t)} \cdot \dod{}{t}} b(t) - {\eval[2]{f(v,t)}_{v=a(t)} \cdot \dod{}{t}} a(t) + {\int_{a(t)}^{b(t)} \dpd{}{t} f(x,t) \dif x}  ,
\end{split}
\]
for all $t \in \mathcal{T}$.  
\end{proof}

\begin{proposition} \label{proposition:Partial Differentiation Under the Integral Sign}
Let $h(x,y,z)$, $\pd{}{z} h(x,y,z)$, and $\md{}{2}{y}{}{z}{} h(x,y,z)$ be defined and continuous on $\mathcal{S} \defeq \mathcal{X} \times \mathcal{Y} \times \mathcal{Z}$, where $\mathcal{X} \defeq [x_1,x_2]$, $\mathcal{Y} \defeq [y_1,y_2]$, and $\mathcal{Z} \defeq [z_1,z_2]$ are subsets of $\R$ with $-\infty < x_1 \leq x_2 < \infty$, $-\infty < y_1 < y_2 < \infty$, and $-\infty < z_1 < z_2 < \infty$. Furthermore, suppose that $\pd{}{y} h(x,y,z)$ exists on $\mathcal{S}$, and  $a, b : \mathcal{Z} \to \mathcal{X}$ are differentiable functions on $\mathcal{Z}$. Then, for~all~$(y,z) \in \mathcal{Y} \times \mathcal{Z}$,  
\begin{equation} \label{equation:Partial Differentiation Under the Integral Sign}
\begin{split}
{\dmd{}{2}{y}{}{z}{} \left( \int_{a(z)}^{b(z)} h(x,y,z) \dif x \right)} & = {\eval[3]{\dpd{}{y} h(x,y,z)}_{x=b(z)} \cdot \dod{}{z}} b(z) - {\eval[3]{\dpd{}{y} h(x,y,z)}_{x=a(z)} \cdot \dod{}{z} a(z)}     \\
& \quad + {\int_{a(z)}^{b(z)} \dmd{}{2}{y}{}{z}{} h(x,y,z) \dif x}  .
\end{split}
\end{equation}
\end{proposition}

\begin{proof}
Firstly, we will compute the partial derivative of $\int_{a(z)}^{b(z)} h(x,y,z) \dif x$ with respect to the variable $z$. By virtue of Lemma~\ref{lemma:Leibniz's integral rule_variable limits} (extended to functions of three variables), we obtain  
\[  
\begin{split}
{\dpd{}{z} \left( \int_{a(z)}^{b(z)} h(x,y,z) \dif x \right)}  & = {\eval[2]{h(x,y,z)}_{x=b(z)} \cdot \dod{}{z}} b(z) - {\eval[2]{h(x,y,z)}_{x=a(z)} \cdot \dod{}{z} a(z)}     \\
& \quad + {\int_{a(z)}^{b(z)} \dpd{}{z} h(x,y,z) \dif x}   ,
\end{split}
\]
for all $(y,z) \in \mathcal{Y} \times \mathcal{Z}$. 

Secondly, by differentiating the previous equation with respect to the variable $y$, the second-order mixed partial derivative of $\int_{a(z)}^{b(z)} h(x,y,z) \dif x$ is given by  
\[
\begin{split}
{\dmd{}{2}{y}{}{z}{} \left( \int_{a(z)}^{b(z)} h(x,y,z) \dif x \right)} & = {\eval[3]{\dpd{}{y} h(x,y,z)}_{x=b(z)} \cdot \dod{}{z}} b(z) - {\eval[3]{\dpd{}{y} h(x,y,z)}_{x=a(z)} \cdot \dod{}{z} a(z)}     \\
& \quad + {\dpd{}{y} \left( {\int_{a(z)}^{b(z)} \dpd{}{z} h(x,y,z) \dif x} \right)}  .
\end{split}
\]

Finally, we conclude the proof by applying Lemma~\ref{lemma:Leibniz's integral rule_constant limits} (extended to functions of three variables) to the last integral.  
\end{proof}

\begin{remark} \label{remark:Uniform convergence assumption}
Proposition~\ref{proposition:Partial Differentiation Under the Integral Sign} remains valid if we replace the assumption that $\md{}{2}{y}{}{z}{} h(x,y,z)$ is continuous on $\mathcal{S}$ with the assumption that ${\int_{a(z)}^{b(z)} \md{}{2}{y}{}{z}{} h(x,y,z) \dif x}$ converges uniformly on $\mathcal{Y} \times \mathcal{Z}$. This is because Lemma~\ref{lemma:Leibniz's integral rule_constant limits} still holds when the continuity of $\pd{}{t} f(x,t)$ on $\mathcal{X} \times \mathcal{T}$ is replaced by the uniform convergence of $\int_{a}^{b} \pd{}{t} {f(x,t) \dif x}$ on $\mathcal{T}$; see \cite{Talvila}~and~\cite[p.~260]{Rogers}. 
\end{remark}

Next, we give modified versions of Lemmas~\ref{lemma:Leibniz's integral rule_constant limits},~\ref{lemma:Leibniz's integral rule_variable limits} and Proposition~\ref{proposition:Partial Differentiation Under the Integral Sign} by relaxing the continuity assumptions (some results now hold \emph{almost everywhere}). Note that the new conditions are weaker than continuity, because if a real-valued function is continuous on a compact set, then it is bounded and so integrable; in addition, there always exist constant functions $\vartheta(x)$,  $\varphi(x,y)$ and $\phi(x)$ satisfying the assumptions in the following results.

\begin{lemma}[Leibniz's integral rule with constant limits of integration -- modified version {\cite[p.~56, Theorem~2.27]{Folland}}]
\label{lemma:Leibniz's integral rule_constant limits_modified version}
Let $f(x,t)$ and $\pd{}{t} f(x,t)$ be defined on $\mathcal{X} \times \mathcal{T}$, where $\mathcal{X} \defeq [x_1,x_2]$ and $\mathcal{T} \defeq [t_1,t_2]$ are subsets of $\R$ with $-\infty < x_1 \leq x_2 < \infty$ and $-\infty < t_1 < t_2 < \infty$. Let $a,b \in \mathcal{X}$ be two constants. Moreover, assume that: 1) the function $f(x,t)$ is integrable, i.e., $\int_{a}^{b} \left| f(x,t) \right| \dif x < \infty$, for all $t \in \mathcal{T}$, and 2) there exists a function $\vartheta(x)$ such that $\left| \pd{}{t} f(x,t) \right| \leq \vartheta(x)$ for all $(x,t) \in \mathcal{X} \times \mathcal{T}$, and $\int_{a}^{b} \vartheta(x) \dif x < \infty$. Then, \eqref{equation:Leibniz's integral rule_constant limits} holds for all $t \in \mathcal{T}$.  
\end{lemma}

\begin{lemma}[Leibniz's integral rule with variable limits of integration -- modified version] 
\label{lemma:Leibniz's integral rule_variable limits_modified version}
Let $f(x,t)$ and $\pd{}{t} f(x,t)$ be defined on $\mathcal{X} \times \mathcal{T}$, where $\mathcal{X} \defeq [x_1,x_2]$ and $\mathcal{T} \defeq [t_1,t_2]$ are subsets of $\R$ with $-\infty < x_1 \leq x_2 < \infty$ and $-\infty < t_1 < t_2 < \infty$. Let $a, b : \mathcal{T} \to \mathcal{X}$ be differentiable functions on $\mathcal{T}$. In addition, assume that: 1) the function $f(x,t)$ is integrable, i.e., $\int_{a(t)}^{b(t)} \left| f(x,t) \right| \dif x < \infty$, for all $t \in \mathcal{T}$, and 2) there exists a function $\vartheta(x)$ such that $\left| \pd{}{t} f(x,t) \right| \leq \vartheta(x)$ for all $(x,t) \in \mathcal{X} \times \mathcal{T}$, and $\int_{a(t)}^{b(t)} \vartheta(x) \dif x < \infty$ for all $t \in \mathcal{T}$. Then, \eqref{equation:Leibniz's integral rule_variable limits} holds for almost all $t \in \mathcal{T}$. 
\end{lemma}

\begin{proof}
This result follows from similar steps as in the proof of Lemma~\ref{lemma:Leibniz's integral rule_variable limits}, but exploiting the fundamental theorem of calculus for Lebesgue integrals~\cite[pp.~105-106, Corollary~3.33 and Theorem~3.35]{Folland} (instead of the classical one) as well as Lemma~\ref{lemma:Leibniz's integral rule_constant limits_modified version} (instead of Lemma~\ref{lemma:Leibniz's integral rule_constant limits}). In particular, the fundamental theorem of calculus for Lebesgue integrals implies that $\displaystyle{\pd{}{u} I(u,v,t) \mathop{=}^{\text{a.e.}} f(u,t)}$ and $\displaystyle{\pd{}{v} I(u,v,t) \mathop{=}^{\text{a.e.}} - f(v,t)}$. Therefore, \eqref{equation:Leibniz's integral rule_variable limits} is true for almost every $t \in \mathcal{T}$.  
\end{proof}

\begin{proposition} \label{proposition:Partial Differentiation Under the Integral Sign_modified version}
Let $h(x,y,z)$, $\pd{}{z} h(x,y,z)$, and $\md{}{2}{y}{}{z}{} h(x,y,z)$ be defined on $\mathcal{S} \defeq \mathcal{X} \times \mathcal{Y} \times \mathcal{Z}$, where $\mathcal{X} \defeq [x_1,x_2]$, $\mathcal{Y} \defeq [y_1,y_2]$, and $\mathcal{Z} \defeq [z_1,z_2]$ are subsets of $\R$ with $-\infty < x_1 \leq x_2 < \infty$, $-\infty < y_1 < y_2 < \infty$, and $-\infty < z_1 < z_2 < \infty$. Suppose that $\pd{}{y} h(x,y,z)$ exists on $\mathcal{S}$, and  $a, b : \mathcal{Z} \to \mathcal{X}$ are differentiable functions on $\mathcal{Z}$. Furthermore, assume that: 1) the function $h(x,y,z)$ is integrable, i.e., $\int_{a(z)}^{b(z)} \left| h(x,y,z) \right| \dif x < \infty$, for all $(y,z) \in \mathcal{Y} \times \mathcal{Z}$, 2) there exists a function $\varphi(x,y)$ such that $\left| \pd{}{z} h(x,y,z) \right| \leq \varphi(x,y)$ for all $(x,y,z) \in \mathcal{S}$, and $\int_{a(z)}^{b(z)} \varphi(x,y) \dif x < \infty$ for all $(y,z) \in \mathcal{Y} \times \mathcal{Z}$, and 3) there is a function $\phi(x)$ such that $\left| \md{}{2}{y}{}{z}{} h(x,y,z) \right| \leq \phi(x)$ for all $(x,y,z) \in \mathcal{S}$, and $\int_{a(z)}^{b(z)} \phi(x) \dif x < \infty$ for all $z \in \mathcal{Z}$. Then, \eqref{equation:Partial Differentiation Under the Integral Sign} holds for almost all $(y,z) \in \mathcal{Y} \times \mathcal{Z}$. 
\end{proposition}

\begin{proof}
In the proof of Proposition~\ref{proposition:Partial Differentiation Under the Integral Sign}, we make the following substitutions: Lemma~\ref{lemma:Leibniz's integral rule_constant limits} $\mapsto$ Lemma~\ref{lemma:Leibniz's integral rule_constant limits_modified version} and Lemma~\ref{lemma:Leibniz's integral rule_variable limits} $\mapsto$ Lemma~\ref{lemma:Leibniz's integral rule_variable limits_modified version}, while the two equations with derivatives are valid almost everywhere now. Observe that the second assumption implies that $\int_{a(z)}^{b(z)} \left| \pd{}{z} h(x,y,z) \right| \dif x < \infty$, for all $(y,z) \in \mathcal{Y} \times \mathcal{Z}$, which is required to apply Lemma~\ref{lemma:Leibniz's integral rule_constant limits_modified version} in the final step.   
\end{proof}

\section{Dirac Delta Function:~Definition and Properties}
\label{appendix:Dirac Delta Function}

Loosely speaking, the Dirac delta function/distribution $\delta(x)$ can be regarded as a function defined on $\R$ that is zero everywhere except at $x=0$, where it is infinite, and also satisfies the following properties: $\delta(-x) = \delta(x) \geq 0$, $\int_{-\infty}^{+\infty} {\delta(t) \dif t} = \int_{-\varepsilon}^{+\varepsilon} {\delta(t) \dif t} = 1$, $\int_{-\infty}^{+\infty} {f(t) \delta(x-t) \dif t} = \int_{x-\varepsilon}^{x+\varepsilon} {f(t) \delta(x-t) \dif t} = f(x)$, and $f(x) \delta(x-y) = f(y) \delta(x-y)$ for all $x,y \in \R$, $\varepsilon \in \R_+$ and for all functions $f(x)$ that are \emph{bounded and piecewise continuous on $\R$}. Note that $\delta(x)$ is not a function in the conventional sense, because no function has these properties. Nevertheless, the Dirac delta function can be rigorously defined as a generalized function or distribution. Strictly speaking, the Dirac delta function $\delta(x)$ is the limit (in the sense of distributions) of a Dirac sequence $\{ \delta_k (x) \}_{k \in \N}$, that is, $\delta(x) = \lim_{k \to \infty} \delta_k(x)$. 

\begin{definition}[{\cite[p.~284]{Lang}}]
A \emph{Dirac sequence} is a sequence of functions $\{ \delta_k (x) \}_{k \in \N}$ that satisfy three properties: P.1) $\delta_k(-x) = \delta_k(x) \geq 0$ for all $k \in \N$ and for all $x \in \R$, \linebreak P.2) $\delta_k (x)$ is continuous on $\R$ and $\int_{-\infty}^{+\infty} {\delta_k(x) \dif x} = 1$ for all $k \in \N$, and P.3) for any $\epsilon, \zeta > 0$, there exists $K \in \N$ such that $\int_{-\infty}^{-\zeta} {\delta_k(x) \dif x} + \int_{\zeta}^{+\infty} {\delta_k(x) \dif x} < \epsilon$ for all integers $k \geq K$.
\end{definition}

\noindent In simple words, P.3 means that the area under the curve $y = \delta_k(x)$ is concentrated around $x=0$ for sufficiently large $k$. For example, the sequence of zero-mean normal/Gaussian distributions $\delta_k(x) = \frac{k}{\sqrt{\pi}} e^{-(kx)^2}$ is a Dirac sequence. 

The following proposition is an important result in analysis about the convolution of a function with a Dirac sequence; the convolution of two functions $f(x)$ and $g(x)$ is defined by $f*g \defeq \int_{-\infty}^{+\infty} {f(t) g(x-t) \dif t}$ and is commutative, that is, $f*g = g*f$.

\begin{proposition}[{\cite[p.~285, Theorem~1.1]{Lang}}]
Suppose that $f(x)$ is a bounded and piecewise-continuous function on $\R$, $\mathcal{D}$ is a compact subset of $\R$ on which $f(x)$ is continuous, and $\{ \delta_k (x) \}_{k \in \N}$ is a Dirac sequence. Then the sequence of functions $\{ f_k (x) \}_{k \in \N}$, where $f_k (x) \defeq f*\delta_k = \int_{-\infty}^{+\infty} {f(t) \delta_k(x-t) \dif t}$, converges to $f(x)$ uniformly on $\mathcal{D}$.
\end{proposition}

Therefore, if $f(x)$ is bounded and piecewise continuous on $\R$, then \linebreak $\int_{-\infty}^{+\infty} {f(t) \delta(x-t) \dif t} = f(x)$ rigorously means that: for each Dirac sequence $\{ \delta_k (x) \}_{k \in \N}$, $\lim_{k \to \infty} {\int_{-\infty}^{+\infty} {f(t) \delta_k(x-t) \dif t}} = f(x)$ uniformly on every compact subset of $\R$ where $f(x)$ is continuous.

\section*{Acknowledgments}
The author thanks the Editor-in-Chief for handling the review process of the paper, and the anonymous referees for their useful comments and suggestions.  


\end{document}